\documentclass[11pt,reqno]{amsart}
\usepackage{amssymb}
\usepackage{natbib}
\bibpunct[, ]{[}{]}{;}{n}{,}{,}

\numberwithin{equation}{section}

\allowdisplaybreaks

\newtheorem{theorem}{Theorem}[section]

\theoremstyle{definition}
\newtheorem{example}[theorem]{Example}

\newtheorem{remark}[theorem]{Remark}

\theoremstyle{remark}

\newenvironment{romenumerate}[1][0pt]{
\addtolength{\leftmargini}{#1}\begin{enumerate}
 }{\end{enumerate}}

\newcounter{oldenumi}
{\setcounter{oldenumi}{\value{enumi}}
\begin{romenumerate} \setcounter{enumi}{\value{oldenumi}}}
{\end{romenumerate}}

\newcounter{thmenumerate}

\newcounter{romxenumerate}   

\newcounter{xenumerate}   

\newcommand{\refT}[1]{Theorem~\ref{#1}}

\newcommand{\refR}[1]{Remark~\ref{#1}}
\newcommand{\refS}[1]{Section~\ref{#1}}
\newcommand{\refSS}[1]{Subsection~\ref{#1}}

\newcommand{\refE}[1]{Example~\ref{#1}}

\newcommand{\refApp}[1]{Appendix~\ref{#1}}

\newcommand{\refand}[2]{\ref{#1} and~\ref{#2}}

\newcommand\marginal[1]{\marginpar{\raggedright\parindent=0pt\tiny #1}}

\begingroup
  \divide\count255 by 60
  \multiply\count255 by -60
  \advance\count255 by \time
  \ifnum \count255 < 10 \xdef\klockan{\the\count1.0\the\count255}
  \else\xdef\klockan{\the\count1.\the\count255}\fi
\endgroup

\newcommand\set[1]{\ensuremath{\{#1\}}}

\newcommand\xpar[1]{(#1)}
\newcommand\bigpar[1]{\bigl(#1\bigr)}
\newcommand\Bigpar[1]{\Bigl(#1\Bigr)}

\newcommand\lrpar[1]{\left(#1\right)}

\def\rompar(#1){\textup(#1\textup)}    

\newcommand\parfrac[2]{\lrpar{\frac{#1}{#2}}}

\def\xexp(#1){e^{#1}}

\newcommand\half{\tfrac12}

\newcommand\punkt[1]{\if.#1\else.\spacefactor1000\fi{#1}}

\newcommand\ie{i.e\punkt}
\newcommand\eg{e.g\punkt}
\newcommand\viz{viz\punkt}
\newcommand\cf{cf\punkt}

\newcommand\ii{\mathrm{i}}

\newcommand\bbR{\mathbb R}
\newcommand\bbQ{\mathbb Q}
\newcommand\bbC{\mathbb C}

\newcommand\bbZ{\mathbb Z}

\newcounter{CC}
\newcounter{cc}

\renewcommand\Re{\operatorname{Re}}
\renewcommand\Im{\operatorname{Im}}

\newcommand\Fix{\operatorname{Fix}}
\newcommand\Gal[2]{\operatorname{Gal}(#1\,{:}\,#2)}
\newcommand\Dis{\operatorname{Dis}}

\newcommand\arccosh{\operatorname{arccosh}}
\newcommand\arcsinh{\operatorname{arcsinh}}
\newcommand\sign{\operatorname{sign}}

\newcommand\ga{\alpha}
\newcommand\gb{\beta}
\newcommand\gd{\delta}
\newcommand\gD{\Delta}
\newcommand\gf{\varphi}
\newcommand\gam{\gamma}

\newcommand\go{\omega}
\newcommand\goo{\omega^2}

\newcommand\gs{\sigma}

\def\[#1]{[\![#1]\!]}

\newcommand\qqq{^{1/3}}

\newcommand\qqc{^{3/2}}

\renewcommand{\=}{:=}

\newcommand\rhs{right-hand side}

\newcommand\bK{\overline K}
\newcommand\hz{\widehat z}
\newcommand\hgb{\widehat \gb}
\newcommand\ArsMagna{\emph{Ars Magna} \cite{ArsMagna}}
\newcommand\ci{\emph{casus irreducibilis}}
\newcommand\pq{\Bigpar{\frac{p}3}^3+\Bigpar{\frac{q}2}^2}
\newcommand\xgam[1]{\ifcase#1\or\gam_0\or\gam_1\or\gam_2\or\gam_3\fi}
\newcommand\tR{\widetilde R}
\newcommand\sx{t}

\newcommand\REM[1]{{\raggedright\texttt{[#1]}\par\marginal{XXX}}}

\newcommand\urladdrx[1]{{\urladdr{\def~{{\tiny$\sim$}}#1}}}

\begin{document}
\title
{Roots of polynomials of degrees 3 and 4}

\date{1 December 2009; revised 17 August 2010}

\author{Svante Janson}
\address{Department of Mathematics, Uppsala University, PO Box 480,
SE-751~06 Uppsala, Sweden}
\email{svante.janson@math.uu.se}
\urladdrx{http://www.math.uu.se/~svante/}

\subjclass[2000]{} 

\begin{abstract} 
We present the solutions of equations of degrees 3 and 4 using Galois
theory and some simple Fourier analysis for finite groups.
\end{abstract}

\maketitle

\section{Introduction}\label{S:intro}

The purpose of this note is to present the solutions of equations of
degrees 3 and 4 
(a.k.a.\ cubic and quartic equations)
in a way connected to Galois theory.
This is, of course, not the historical path; the solutions were found
by del Ferro, Tartaglia, Cardano (Cardan) and Ferrari in the 16th century, see
\refApp{Ahistory}, 
about 300 years before Galois
theory was created. But in retrospect, Galois theory yields the
solutions rather easily. 
More precisely, we will see below that Galois theory, together with
some simple Fourier analysis for (small) finite Abelian groups, suggests the
crucial constructions in the solutions; however, all calculations are
easily verified directly, and we do not really need any results from
Galois theory (or Fourier analysis)
for the solution. Nevertheless, we find it instructive to use Galois
theory as much as possible in order to motivate the constructions.
(See also \cite[Section 8.8]{vdW1} for a similar, but not identical,
treatment.) 
The Galois theory used here can be found in \eg\ \cite{Bew}, \cite{Cox},
\cite{Garling}, \cite{Grillet} or \cite{vdW1}. 

The appendices contain comments on the history of the problem, other solutions,
and (\refApp{Areal}) the complications that may arise when we work with real
numbers instead of complex.

\begin{remark}
  In contrast, full use of Galois theory is needed to show the
  impossibility of similar formulas for solutions of equations of
  degree 5 or more. This will not be discussed here; see instead \eg\
\cite{Bew}, \cite{Cox}, \cite{Garling},  \cite{Grillet} or \cite{vdW1}. 
\end{remark}

We let throughout  $K$ be a field with characteristic 0. (Actually,
everything in 
this note is valid also for a field $K$ of positive characteristic
$p\neq2,3$. However, the cases when the characteristic is 2 or 3 are
different since we divide by 2 and 3 in the formulas below; there are
also problems with separability in these cases.)

The roots of a polynomial in $K[x]$ are, in general, not elements of
$K$, so we will work in some unspecified extension of $K$. This
extension could be the algebraic closure $\bK$ of $K$ or some other
algebraically closed field containing $K$; in particular, if $K=\bbQ$ or
another subfield of $\bbC$ (as the reader may assume for simplicity), we can
work in $\bbC$. 

For simplicity we consider monic polynomials only.
(It is trivial to reduce to this case by dividing by the leading coefficient.)

\section{Polynomials of degree 3}\label{S3}

Let $f(x)=x^3+bx^2+cx+d$, with $b,c,d\in K$, be a polynomial of degree
3, and let $\ga_1,\ga_2,\ga_3$ be its roots in some extension of $K$.
Thus
\begin{equation}\label{f}
  f(x)=x^3+bx^2+cx+d=(x-\ga_1)(x-\ga_2)(x-\ga_3).
\end{equation}

It is convenient to make the translation $x=y-b/3$, converting $f(x)$
into
\begin{equation}\label{fg}
  g(y)\=f(y-b/3)=y^3+py+q
\end{equation}
(without second degree term) for some $p,q\in K$.
(Such polynomials, without the second highest degree term, are 
called \emph{reduced} or sometimes \emph{depressed}.)
Thus $g$ has the roots $\gb_1,\gb_2,\gb_3$ with $\gb_i=\ga_i+b/3$, so
$\ga_i=\gb_i-b/3$, $i=1,\dots,3$. Hence,
\begin{equation}\label{g}
  g(y)=y^3+py+q=(y-\gb_1)(y-\gb_2)(y-\gb_3).
\end{equation}
Consequently, identifying coefficients,
\begin{align}
  \gb_1+\gb_2+\gb_3&=0, \label{gbo}\\
  \gb_1\gb_2+\gb_1\gb_3+\gb_2\gb_3&=p, \label{gbp}\\
  \gb_1\gb_2\gb_3&=-q, \label{gbq}
\end{align}

\begin{remark}
  Explicitly,
  \begin{align}
	p&=c-\tfrac13b^2, \label{p}\\
q&=d-\tfrac13 bc+ \tfrac2{27}b^3. \label{q}
  \end{align}
\end{remark}

The polynomials $f$ and $g$ have the same discriminant
\begin{equation}\label{gD0}
  \gD\=\Dis(f)=\Dis(g)
\=\prod_{1\le i<j\le 3}(\ga_i-\ga_j)^2
=\prod_{1\le i<j\le 3}(\gb_i-\gb_j)^2.
\end{equation}
Since $\gD$ is a symmetric polynomial in $\ga_1,\ga_2,\ga_3$ (or
$\gb_1,\gb_2,\gb_3)$, it can be written as a polynomial in the
coefficients of $f$ or $g$. A well-known calculation yields, see \eg{}
\cite{SJN5}, 
\begin{equation}\label{gD}
\gD
  = b^2c^2-4 c^3-4 b^3d+18bcd-27d^2
= -4 p^3-27q^2.
\end{equation}

We also define the square root of $\gD$:
\begin{equation}\label{gd}
  \gd\=
\prod_{1\le i<j\le 3}(\ga_i-\ga_j)
=\prod_{1\le i<j\le 3}(\gb_i-\gb_j)
=\sqrt{\gD}.
\end{equation}
Note that while $\gD$ is independent of the ordering of the roots,
the sign of $\gd$ may change if we permute $\ga_1,\ga_2,\ga_3$.
More precisely, the sign is preserved by an even permutation but is
changed by an odd permutation.

Let $E=K(\ga_1,\ga_2,\ga_3)=K(\gb_1,\gb_2,\gb_3)$ be the splitting field
of $f$, or $g$, over $K$, and let $G\=\Gal EK$ be the Galois group
of the extension $E\supseteq K$. 
The elements of the Galois group $G$ permute the roots $\ga_i$ (or
$\gb_i$), and $G$ may be regarded as a subgroup of the symmetric group $S_3$.
As said above, $\gs(\gd)=\gd$ if $\gs\in G$ is an even permutation,
while $\gs(\gd)=-\gd$ if $\gs$ is odd.
Since $K \subseteq K(\gd)\subseteq E$, $E $ is a Galois extension of
$K(\gd)$, and the Galois group $\Gal{E}{K(\gd)}$ is the subgroup of
$G$ fixing $\gd$:
\begin{equation}
  \Gal E{K(\gd)}=\set{\gs\in G:\gs(\gd)=\gd}
=\set{\gs\in G:\gs\text{ is even}}
=G\cap A_3
\end{equation}
(where $A_3$ is the alternating group consisting of all even permutations in
$S_3$),
at least if $\gd\neq0$, or equivalently $\gD\neq0$, which is
equivalent to $f$ separable (\ie, $f$ has no multiple roots in
$\bK$).
In particular, if $f$ is irreducible, in which case $G$ is transitive
and thus $G=S_3$ or $A_3$, $\Gal E{K(\gd)}=A_3$, which
is the cyclic group $C_3$.
Moreover, in this case,
 $\Gal E{K(\gd)}=A_3\cong C_3$ acts on
the vectors $(\ga_1,\ga_2,\ga_3)$ and
$(\gb_1,\gb_2,\gb_3)$ by cyclic permutations; equivalently, if we regard
these vectors as functions on $\bbZ_3\cong C_3$,
 $\Gal E{K(\gd)}\cong C_3$ acts by translations in $C_3$.
This suggests using
Fourier analysis, or equivalently group representation theory, for
$A_3\cong C_3$.
(For Fourier analysis on finite Abelian groups, see \eg{}
\cite{Frazier}; in this case, the Fourier transform is often called the
\emph{discrete Fourier transform}.
The more complicated theory of group representations for general 
finite groups is treated by \cite{Serre}.)

\begin{remark}
The method below was given by \citet{Lagrange} in \citeyear{Lagrange}, 
as part of his study of equations of higher degree, 
see \cite[Sections 8.3,  12.1 and p.~14]{Cox}. The method 
is thus some decades older than both Galois theory and Fourier analysis.
In this context, the Fourier transforms $u$ and $v$ 
in \eqref{u}--\eqref{v} below
(or rather $3u$ and $3v$)
are known as the \emph{Lagrange resolvents} for the equation.
(They were also used earlier by B\'ezout and Euler \cite[p.~46]{Bew}, 
and at the same time by Vandermonde \cite{vdW}, but Lagrange made a
profound use of them.)  
Lagrange and others used this method to study equations of arbitrary degree,
see \cite[Sections 8.3 and 12.1]{Cox}
and, for example,
the solutions by Vandermonde
and Malfatti of some quintic equations givin in \cite[Chapters 7--8]{Bew}.
This is an important forerunner of Galois theory.
In retrospect, the Lagrangre resolvents can perhaps also be seen as the
beginning of discrete Fourier analysis. 
\end{remark}

We assume, for simplicity, that $K\subseteq\bbC$, and we then
define 
\begin{equation}
 \go\=\exp(2\pi\ii/3)=-\frac12+\frac{\sqrt3}2\ii,  
\end{equation}
a primitive third root of unity.
Note that
\begin{align}
  \go^3=1, &&& 1+\go+\go^2=0,
\end{align}
which will be used repeatedly below without further comment.
\begin{remark}
  For a general field $K$, not necessarily contained in $\bbC$, we can
let $\go$ be a primitive third root of unity in $\bK$. It is easily
verified that the formulas below make sense, and are correct, in
$\bK$, so the result holds in full generality.
\end{remark}

We then define the Fourier transform 
of a vector $(z_1,z_2,z_3)$
(regarded as a function on $\bbZ_3\cong C_3$) as $(\hz_1,\hz_2,\hz_3)$,
with 
\begin{equation}
 \hz_k\=\tfrac13\bigpar{z_1+\go^{-(k-1)} z_2+\go^{-2(k-1)} z_3}, 
\end{equation}
and note the Fourier inversion formula,
which in this case is easily verified directly, 
\begin{equation}\label{finv}
 z_k= \hz_1+\go^{k-1} \hz_2+\go^{2(k-1)} \hz_3. 
\end{equation}
Hence, 
if we define
\begin{align}
  u\=\tfrac13(\ga_1+\go^2\ga_2+\go\ga_3)
=\tfrac13(\gb_1+\go^2\gb_2+\go\gb_3),\label{u}\\
  v\=\tfrac13(\ga_1+\go\ga_2+\go^2\ga_3)
=\tfrac13(\gb_1+\go\gb_2+\go^2\gb_3),\label{v}
\end{align}
and note that
\begin{align}
  -\tfrac13b&=\tfrac13(\ga_1+\ga_2+\ga_3),
\\
  0&=\tfrac13(\gb_1+\gb_2+\gb_3),
\end{align}
we see that
the Fourier transforms of the vectors $(\ga_1,\ga_2,\ga_3)$ and
$(\gb_1,\gb_2,\gb_3)$ are
$(-\frac13b,u,v)$ and $(0,u,v)$, respectively.
Consequently, the inversion formula \eqref{finv} yields
\begin{align}
  \ga_1&=-\tfrac13b+u+v,\label{ga1}\\
  \ga_2&=-\tfrac13b+\go u+\go^2 v,\label{ga2}\\
  \ga_3&=-\tfrac13b+\go^2 u+\go v,\label{ga3}
\intertext{and, equivalently,}
  \gb_1&=u+v,\label{gb1}\\
  \gb_2&=\go u+\go^2 v,\label{gb2}\\
  \gb_3&=\go^2 u+\go v.\label{gb3}
\end{align}

To solve the equation $f(x)=0$, it thus suffices to find $u$ and $v$.

The objects $u$ and $v$ are elements of the field $E(\go)$, which is
the splitting field of $f$ (or $g$) over $K(\go)$. It is thus a Galois
extension of $K(\go)$, and also of the intermediate field
$K(\gd,\go)$.
An element of the Galois group $\Gal{E(\go)}{K(\go)}$ maps $E$ into
itself (because it fixes $K$ and $E$ is a normal extension of $K$),
and thus its restriction to $E$ is an element of $\Gal EK$. This
defines a group homomorphism 
$\Gal{E(\go)}{K(\go)}\to\Gal EK$, which is injective because $E(\go)$
is generated by $E$ and $K(\go)$; thus we can regard 
$\Gal{E(\go)}{K(\go)}$ as a subgroup of $\Gal EK$. 
Similarly,
$\Gal{E(\go)}{K(\gd,\go)}$ is a subgroup of $\Gal E{K(\gd)}$. 

Let $H\=\Gal{E(\go)}{K(\gd,\go)}$. Then $H\subseteq \Gal
E{K(\gd)}\subseteq A_3$, so 
if $H$ is not trivial, then $H=A_3$ and $H$ is
generated by a cyclic permutation $\gs$ with $\gs(\ga_k)=\ga_{k+1}$
(with indices modulo 3). Then, by \eqref{u}--\eqref{v}, 
$\gs(u)=\go u$ and $\gs(v)=\go^2 v$. Consequently,
$\gs(u^3)=u^3$, and $\gs(v^3)=v^3$. This implies that $u^3$ and $v^3$
are fixed by the Galois group $H$, and thus $u^3$ and $v^3$ belong to
the fixed field $\Fix_{E(\go)}(H)=K(\gd,\go)$.
We can find them as follows,  using
$x^3-1=(x-1)(x-\go)(x-\go^2)$ and thus
$x^3-y^3=(x-y)(x-\go y)(x-\go^2y)$:
First, by \eqref{gb1}--\eqref{gb3} and \eqref{gbq},
\begin{equation}\label{u3+v3}
  \begin{split}
u^3+v^3=(u+v)(u+\go v)(u+\go^2 v)
=\gb_1\gb_2\gb_3
=-q.	
  \end{split}
\end{equation}
Next, by \eqref{u}--\eqref{v},
\begin{align}
  u-v&=\tfrac13(\go^2-\go)(\ga_2-\ga_3)\\
  u-\go v&=\tfrac13(1-\go)(\ga_1-\ga_3)\\
  u-\go^2v&=\tfrac13(1-\go^2)(\ga_1-\ga_2)
\end{align}
and thus, using \eqref{gd},
\begin{equation}\label{sw}
  \begin{split}
u^3-v^3&=(u-v)(u-\go v)(u-\go^2 v)
=\frac{-\sqrt 3\,\ii}{9}(\ga_1-\ga_2)(\ga_1-\ga_3)(\ga_2-\ga_3)
\\&
=-\frac{\sqrt 3\,\ii}{9}\gd
=\sqrt{-\frac{\gD}{27}}.
  \end{split}
\raisetag{\baselineskip}
\end{equation}

\begin{remark}\label{Ruv}
The choice of square root in \eqref{sw} is not important, since a change of
sign of it is equivalent to an interchange of $u$ and $v$, which just
permutes the roots $\ga_2$ and $\ga_3$ ($\gb_2$ and $\gb_3$).)
This reflects the fact that if the Galois group $\Gal{E(\go)}{K(\go)}$
contains an odd 
permutation $\tau$, then $\tau(u)=\go^j v$ and  $\tau(v)=\go^j u$ for
some $j=0,1,2$; thus $\tau(u^3)=v^3$ and $\tau(v^3)=u^3$.
  \end{remark}

We thus find, recalling \eqref{gD},
\begin{align} \label{u3}
  u^3&=\frac{-q+\sqrt{-\gD/27}}2
=-\frac q2+\sqrt{\frac{-\gD}{108}}
=-\frac q2+\sqrt{\pq},
\\  \label{v3}
  v^3&=\frac{-q-\sqrt{-\gD/27}}2
=-\frac q2-\sqrt{\frac{-\gD}{108}}
=-\frac q2-\sqrt{\pq}.
\end{align}
We then find $u$ and $v$ by taking cube roots. In order to choose the
right roots, we also compute, from \eqref{u}--\eqref{v} and
\eqref{gbo}--\eqref{gbp},
\begin{equation}
  \label{uv}
  \begin{split}
uv&=
\tfrac19\bigpar{\gb_1^2+\gb_2^2+\gb_3^2+(\go+\go^2)(\gb_1\gb_2+\gb_1\gb_3+\gb_2\gb_3)}
\\&
=
\tfrac19\bigpar{(\gb_1+\gb_2+\gb_3)^2-3(\gb_1\gb_2+\gb_1\gb_3+\gb_2\gb_3)}
\\&
=-\tfrac13 p.	
  \end{split}
\end{equation}
If we replace $u$ by the alternative cube root $\go u$ or $\goo u$, we
thus have to replace $v$ by $\goo v$ or $\go v$, respectively, which by 
\eqref{ga1}--\eqref{ga3} and \eqref{gb1}--\eqref{gb3} yields a cyclic
permutation of the roots $\ga_1,\ga_2,\ga_3$ or $\gb_1,\gb_2,\gb_3$.
We summarize:

\begin{theorem}[Cardano's formula] \label{T3}
The roots of $g(y)=y^3+py+q$ are given by
\begin{equation}\label{t3a}
\sqrt[3]{-\frac q2+\sqrt{\pq}}
+
\sqrt[3]{-\frac q2-\sqrt{\pq}},
\end{equation}
  where the two square roots are chosen to be the same, and the two
  cube roots are chosen such that their product is $-p/3$; this gives
  $3$ choices for the cube roots, which gives the $3$ roots of $g(y)=0$.
(In the exceptional case $p=q=0$, the only root $0$ is counted thrice.)

Equivalently, the roots of $f(x)=x^3+bx^2+cx+d$ are given by 
\begin{equation}\label{t3b}
  \begin{split}
-\frac b3+  \sqrt[3]{-\frac q2+\sqrt{\pq}}
&+
\sqrt[3]{-\frac q2-\sqrt{\pq}}
\\
= -\frac b3 +
  \sqrt[3]{-\frac q2+\sqrt{\frac{-\gD}{108}}}
&+
\sqrt[3]{-\frac q2-\sqrt{\frac{-\gD}{108}}},
  \end{split}
\end{equation}
with $p$ and $q$ given by \eqref{p}--\eqref{q}, and $\gD$ given by \eqref{gD}.
\end{theorem}

\begin{remark}
  This formula is known as Cardano's formula since it was first published by
  Cardano in \ArsMagna, although it is attributed by him to Scipione del
  Ferro, see   \refApp{Ahistory}. 
\end{remark}

\begin{remark}\label{R30}
  The case $p=q=0$ is exceptional because then (and only then)
  $u^3=v^3=0$. This is the trivial case when $f$ and $g$ are cubes
$(x+b/3)^3$ and $y^3$  and thus have triple roots $-b/3$ and 0, respectively.

The case with a double (but not triple) root are handled correctly by
\refT{T3}.
This is the case when $\gD=0$ (but not $p=q=0$), and thus $u^3=v^3$
($\neq0$).
We can find a cube root $u=v$ with $uv=u^2=-p/3$, and then the other
eligible pairs of cube roots are $(\go u,\goo u)$ and $(\goo u,\go
u)$,
yielding the roots $\gb_1=2u$, $\gb_2=\gb_3=-u$, and thus
$\ga_1=-b/3+2u$, $\ga_2=\ga_3=-b/3-u$.

Similarly, there are no problems in the case when $u^3$ or $v^3$ is 0,
but not both. This happens, by \eqref{u3}--\eqref{v3} and \eqref{uv},
when $p=0$ but $q\neq0$. Choosing the square root such that $v^3=0$,
we have $u^3=-q$; the polynomial $g(y)$ equals $y^3+q$ which has the three
roots $u$, $\go u$, $\goo u$.
\end{remark}

\begin{remark}
By \eqref{u3+v3} and \eqref{uv}, which implies $u^3v^3=-p^3/27$, $u^3$
and $v^3$ are the roots of the \emph{quadratic resolvent}
\begin{equation}\label{quadres}
  r(x)\=x^2+qx-p^3/27\in K[x].
\end{equation}
Note that the quadratic
resolvent has discriminant, by \eqref{sw},
\begin{equation}
  \Dis(r)\=(u^3-v^3)^2=-\frac{\gD}{27}=-\frac1{27}\Dis(f).
\end{equation}
\end{remark}

\section{Polynomials of degree 4}\label{S4}

Let $f(x)=x^4+bx^3+cx^2+dx+e$, with $b,c,d,e\in K$, be a polynomial of degree
4, and let $\ga_1,\ga_2,\ga_3,\ga_4$ be its roots in some extension of $K$.
Thus
\begin{equation}\label{f4}
  f(x)=x^4+bx^3+cx^2+dx+e=(x-\ga_1)(x-\ga_2)(x-\ga_3)(x-\ga_4).
\end{equation}

It is convenient to make the translation $x=y-b/4$, converting $f(x)$
into
\begin{equation}\label{fg4}
  g(y)\=f(y-b/4)=y^4+py^2+qy+r
\end{equation}
(without third degree term) for some $p,q,r\in K$.
Thus $g$ has the roots $\gb_1,\gb_2,\gb_3,\gb_4$ with $\gb_i=\ga_i+b/4$, so
$\ga_i=\gb_i-b/4$, $i=1,\dots,4$. Hence,
\begin{equation}\label{g4}
  g(y)=y^4+py^2+qy+r=(y-\gb_1)(y-\gb_2)(y-\gb_3)(y-\gb_4).
\end{equation}


The polynomials $f$ and $g$ have the same discriminant
\begin{equation}
  \gD\=\Dis(f)=\Dis(g)=\prod_{1\le i<j\le 4}(\ga_i-\ga_j)^2
=\prod_{1\le i<j\le 4}(\gb_i-\gb_j)^2.
\end{equation}
Since $\gD$ is a symmetric polynomial in $\ga_1,\dots,\ga_4$ (or
$\gb_1,\dots,\gb_4)$, it can be written as a polynomial in the
coefficients of $f$ or $g$. A well-known calculation yields, see \eg{}
\cite{SJN5}, 
\begin{align}
  \gD
&
= 
{b}^{2}{c}^{2}{d}^{2}-4\,{b}^{2}{c}^{3}e-4\,{b}^{3}{d}^{3}+18\,{b}^{3}
cde-27\,{b}^{4}{e}^{2}-4\,{c}^{3}{d}^{2}
\notag \\&\qquad
+16\,{c}^{4}e+18\,bc{d}^{3}-80
\,b{c}^{2}de-6\,{b}^{2}{d}^{2}e+144\,{b}^{2}c{e}^{2}
\notag \\&\qquad
-27\,{d}^{4}
+144\,c{d}^{2}e-128\,{c}^{2}{e}^{2}-192\,bd{e}^{2}+256\,{e}^{3}
\label{gD4a} \\&
= 
-4\,{p}^{3}{q}^{2}
-27\,{q}^{4}
+16\,{p}^{4}r+144\,p{q}^{2}r-128\,{p}^{2
}{r}^{2}+256\,{r}^{3}
. \label{gD4b}
\end{align}

We also define the square root of $\gD$:
\begin{equation}\label{gd4}
  \gd\=
\prod_{1\le i<j\le 4}(\ga_i-\ga_j)
=\prod_{1\le i<j\le 4}(\gb_i-\gb_j)
=\sqrt{\gD}.
\end{equation}
Again,
the sign of $\gd$ may change if we permute $\ga_1,\dots,\ga_4$;
the sign is preserved by an even permutation but is
changed by an odd permutation.

Let $E=K(\ga_1,\dots,\ga_4)=K(\gb_1,\dots,\gb_4)$ be the splitting field
of $f$, or $g$, over $K$, and let $G\=\Gal EK$ be the Galois group
of the extension $E\supseteq K$. 
The elements of the Galois group $G$ permute the roots $\ga_i$ (or
$\gb_i$), and $G$ may be regarded as a subgroup of $S_4$.

$S_4$ has a normal subgroup $V$ consisting of the 4 permutations
$\iota$ (identity) and $(12)(34)$, $(13)(24)$, $(14)(23)$. Thus $G$
has a normal subgroup $G\cap V$. Let the fixed field of $G\cap V$ be
$F$. Then $F$ is a Galois extension of $K$ with Galois group $G/G\cap
V\subseteq S_4/V\cong S_3$.

Fourier analysis on $V$ is especially simple because every element has
order 1 or 2, and thus every character is $\pm1$ (again, see \eg{}
\cite{Frazier} or \cite{Serre}). We identify functions on $V$ by
vectors $(z_1,z_2,z_3,z_4)$, with $z_1$ the value at $\iota$, and
define the Fourier transform of 
$(z_1,z_2,z_3,z_4)$ as $(\hz_1,\hz_2,\hz_3,\hz_4)$ with 
\begin{align}
  \hz_1&\=\tfrac12(z_1+z_2+z_3+z_4),\\
  \hz_2&\=\tfrac12(z_1+z_2-z_3-z_4),\\
  \hz_3&\=\tfrac12(z_1-z_2+z_3-z_4),\\
  \hz_4&\=\tfrac12(z_1-z_2-z_3+z_4).
\end{align}
For $V$, with our chosen normalization, the Fourier inversion formula
takes the especially simple form $\widehat{\widehat z}=z$, \ie, the
Fourier transform is its own inverse. (This is easily verified directly.)

Since $\gb_1+\gb_2+\gb_3+\gb_4=0$, the Fourier coefficient $\hgb_1=0$.
For convenience, we shift the indices and define $\gam_i\=\hgb_{i+1}$,
$i=0,1,2,3$, where thus $\gam_0=0$.  
The Fourier transforms of $(\ga_1,\ga_2,\ga_3,\ga_4)$ and
$(\gb_1,\gb_2,\gb_3,\gb_4)$ are thus $(-\frac12 b,\xgam2,\xgam3,\xgam4)$ 
and $(0,\xgam2,\xgam3,\xgam4)$, where
\begin{align}
  \xgam2&\=\hgb_2\=\tfrac12(\ga_1+\ga_2-\ga_3-\ga_4)=\gb_1+\gb_2=-\gb_3-\gb_4,
\label{gam1}\\
  \xgam3&\=\hgb_3\=\tfrac12(\ga_1-\ga_2+\ga_3-\ga_4)=\gb_1+\gb_3=-\gb_2-\gb_4,
\label{gam2}\\
  \xgam4&\=\hgb_4\=\tfrac12(\ga_1-\ga_2-\ga_3+\ga_4)=\gb_1+\gb_4=-\gb_2-\gb_3.
\label{gam3}
\end{align}

Permutations in $V$ act on the vectors (regarded as functions on $V$)
by translations in $V$, and thus on the Fourier transforms by
multiplying by characters, which are $\pm1$. In other words,
permutations in $V$ act on $\xgam2$, $\xgam3$ and $\xgam4$ by
multiplying by $\pm1$ (as is easily seen directly from
\eqref{gam1}--\eqref{gam3}). 
Consequently, if we define
\begin{align}
u&\=  \xgam2^2=(\gb_1+\gb_2)^2=(\gb_3+\gb_4)^2, \label{u4}\\
v&\=  \xgam3^2=(\gb_1+\gb_3)^2=(\gb_2+\gb_4)^2, \label{v4}\\
w&\=  \xgam4^2=(\gb_1+\gb_4)^2=(\gb_2+\gb_3)^2, \label{w4}
\end{align}
then $u$, $v$ and $w$ are fixed by $V\cap G$, and the thus belong to
the fixed field $F$.
We can easily find them explicitly. 
If $\gs$ is any element of the Galois group $G$, then $\gs$ permutes
$\gb_1,\dots,\gb_4$, and it follows from \eqref{u4}--\eqref{w4} that
$\gs$ permutes $u,v,w$. Hence any symmetric polynomial in $u,v,w$ is
fixed by every $\gs\in G$, and thus it belongs to $K$.
In particular, this applies to the coefficients of 
the 
polynomial
$  R(x)\=(x-u)(x-v)(x-w)$, which thus has coefficients in $K$.
Calculations yield the explicit formulas
\begin{align}
  u+v+w&=-2p, \label{u+v+w}\\
uv+uw+vw&=p^2-4r,\\
uvw&=q^2.  \label{uvw}
\end{align}
Hence, $u$, $v$, $w$ are the three roots of 
the \emph{cubic resolvent} 
\begin{equation}\label{R}
  R(x)\=(x-u)(x-v)(x-w)=x^3+2px^2+(p^2-4r)x-q^2 \in K[x].
\end{equation}

\begin{remark}
  Note that 
$u-v=(\gb_1-\gb_4)(\gb_2-\gb_3)$,
$u-w=(\gb_1-\gb_3)(\gb_2-\gb_4)$,
$v-w=(\gb_1-\gb_2)(\gb_3-\gb_4)$.
Hence the discriminant $(u-v)^2(u-w)^2(v-w)^2$ of the cubic resolvent
$R$ equals the
discriminant of $g$ or $f$ given by \eqref{gD4a}--\eqref{gD4b}.
\end{remark}

Having found $u,v,w$, we take their square roots to find
$\xgam2,\xgam3,\xgam4$. 
By \eqref{uvw}, $\xgam2\xgam3\xgam4=\pm q$.
In fact, using \eqref{gam1}--\eqref{gam3} and $\gb_1+\gb_2+\gb_3+\gb_4=0$,
\begin{equation*}
  \begin{split}
&  \xgam2\xgam3\xgam4+q
=   \xgam2\xgam3\xgam4-
 (\gb_1\gb_2\gb_3+\gb_1\gb_2\gb_4+\gb_1\gb_3\gb_4+\gb_2\gb_3\gb_4)
\\
&\qquad=-(\gb_3+\gb_4)(\gb_2+\gb_4)(\gb_2+\gb_3)
+(\gb_2+\gb_3+\gb_4)(\gb_2\gb_3+\gb_2\gb_4+\gb_3\gb_4)
\\ &\qquad\qquad {}-\gb_2\gb_3\gb_4
\\ &\qquad 
=0.	
  \end{split}
\end{equation*}
Hence,
\begin{equation}\label{gam234}
  \xgam2\xgam3\xgam4=-q,  
\end{equation}
which provides the information we need on the signs of 
$\xgam2,\xgam3,\xgam4$. 
We then find $\gb_1,\dots,\gb_4$ by taking the (inverse) Fourier
transform of $(0,\xgam2,\xgam3,\xgam4)$. 
We summarize the resulting algorithm:
\begin{theorem}\label{T4}
  Let $g(y)=y^4+py^2+qy+r$, and define the cubic resolvent $R(x)$ by
  \eqref{R}.
Let the roots of $R$ by $u,v,w$ (for example found by \refT{T3}),
and let $\xgam2\=\sqrt u$, $\xgam3\=\sqrt v$, $\xgam4\=\sqrt w$,
where we choose the signs so that $\xgam2\xgam3\xgam4=-q$.
Then the roots of $g$ are given by
\begin{align}
  \gb_1&=\tfrac12(\xgam2+\xgam3+\xgam4),\label{x1}\\
  \gb_2&=\tfrac12(\xgam2-\xgam3-\xgam4),\\
  \gb_3&=\tfrac12(-\xgam2+\xgam3-\xgam4),\\
  \gb_4&=\tfrac12(-\xgam2-\xgam3+\xgam4).\label{x4}
\end{align}
The roots of $f(x)=x+bx^3+cx^2+dx+e$ are $\ga_i=\gb_i-b/4$,
$i=1,\dots,4$, where $g(y)\=f(y-b/4)$.
\end{theorem}

Note that changing the signs of some of $\xgam2,\xgam3,\xgam4$ while still
preserving \eqref{gam234} (\ie, changing the sign of exactly two of them),
just yields a permutation of $\beta_1,\dots,\gb_4$.

\begin{remark}
  \label{REuler}
The formulas \eqref{x1}--\eqref{x4} were given by Euler \cite[\S 5]{Euler}
in 1733. Euler's motivation was different. 
For the cubic,
Cardano's formula is $\sqrt[3]{U}+\sqrt[3] V$ where $U\=u^3$ and
$V\=v^3$ are roots of the quadratic resolvent \eqref{quadres}.
Further, for a quadratic $x^2=a$ there is 
the trivial formula $\sqrt{a}$.
Hence, Euler sought by analogy a formula for the roots of a quartic in the form
$\sqrt{A}+\sqrt{B}+\sqrt{C}$, and found a cubic equation for $A,B,C$ by 
substituting in $g(y)=0$, see \cite{Euler} for details.
(In our notation, $A=u/4$, $B=v/4$ and $C=v/4$, so Euler's cubic equation
is our $R(4x)=0$; the difference from our cubic resolvent equation $R(u)=0$
is thus only a 
trivial matter of normalization.) 

Euler \cite[\S\S 6--8]{Euler} proceeded to write the solution as 
$\sqrt[4]E+\sqrt[4]F+\sqrt[4]G$ (with $E=A^2$, $F=B^2$, $G=C^2$), and found
another cubic equation satisfied by $E,F,G$.
Euler conjectured that similar formulas existed for higher degrees too, and
in particular that the roots of a fifth degree equation could be found as
$\sqrt[5]A+\sqrt[5]B+\sqrt[5]C+\sqrt[5]D$, where $A,B,C,D$ were the roots of
some fourth degree resolvent; however, he could not find such a resolvent.
Of course, 
we know that Euler's conjecture cannot hold, since
100 years later it was proved by Abel and Galois that in
general there is no solution by radicals for a fifth degree equation.
\end{remark}

\begin{remark}\label{Rbb}
  A simple calculation yields, by \eqref{x1}--\eqref{x4} and \eqref{u+v+w},
  \begin{equation}
	\begin{split}
\gb_1\gb_2+\gb_3\gb_4
&  =\frac14\bigpar{\xgam2^2-(\xgam3+\xgam4)^2  +\xgam2^2-(\xgam3-\xgam4)^2}
\\&
  =\frac14\bigpar{2\xgam2^2-2\xgam3^2  -2\xgam4^2}	  
  =\frac12\bigpar{u-v-w}
=u+p
	\end{split}
  \end{equation}
and similarly
\begin{align}
  \gb_1\gb_3+\gb_2\gb_4&=v+p,
\\
  \gb_1\gb_4+\gb_2\gb_3&=w+p.
\end{align}
Hence, the roots of the cubic resolvent are the three values of
$\gb_i\gb_j+\gb_k\gb_l-p$ for different permutations $ijkl$ of $1234$.

For the roots $\ga_i$ of $f$ we have,
 since $\ga_i=\gb_i-b/4$ and $\gb_1+\gb_2+\gb_3+\gb_4=0$,
\begin{equation}
  \ga_1\ga_2+\ga_3\ga_4=\gb_1\gb_2+\gb_3\gb_4+\frac{b^2}{8}=u+p+\frac{b^2}8,
\end{equation}
and similarly 
$\ga_1\ga_3+\ga_2\ga_4=v+p+b^2/8$,
 $ \ga_1\ga_4+\ga_2\ga_3=w+p+b^2/8$.
\end{remark}

\begin{remark}
  \label{Rbb1}
Another method to solve the quartic equation $x^4+bx^3+cx^2+dx+e=0$, 
also due to \citet{Lagrange}, is
to form (\cf{} \refR{Rbb})
\begin{align}
  s_1&\=\ga_1\ga_2+\ga_3\ga_4,&
  s_2&\=\ga_1\ga_3+\ga_2\ga_4,&
  s_3&\=\ga_1\ga_4+\ga_2\ga_3,
\end{align}
and the cubic polynomial 
$\tR(z)=(z-s_1)(z-s_2)(z-s_3)$ with these as roots.
This polynomial can be expressed in the coefficients of the equation as,
see \cite[Section 12.1]{Cox},
\begin{equation}\label{tR}
  \tR(z)=z^3-cz^2+(bd-4e)z-d^2-b^2e+4ce.
\end{equation}
By \refR{Rbb} and \eqref{R}, $\tR(z)=R(z-p-b^2/8)$, so this is our usual
cubic resolvent in disguise.

Having found $s_1,s_2,s_3$ by solving the resolvent equation $\tR(s)=0$, one
notes, for $k=1,2,3$, 
\begin{equation}
  (2\gam_k)^2-4s_k=b^2-4c,
\end{equation}
and thus
\begin{equation}
  \gam_k=\pm\sqrt{s_k-c+b^2/4},
\end{equation}
which, recalling \eqref{gam234}, yields
the roots by \eqref{x1}--\eqref{x4}.
\end{remark}

\appendix

\section{Real cubic equations}\label{Areal}

Assume that $f(x)=x^3+ax^2+bx+c$ is a polynomial of degree 3 with real
coefficients.

Then $f$ has always 3 complex roots (not necessarily distinct, and as always
given by \refT{T3}),
but the number of real roots may be smaller. 
The following theorem shows
that the number of real roots is 1 or 3, and that the discriminant
discriminates between the possible cases.

\begin{theorem}\label{TR3}
  Let $\gD$ be the discriminant of $f$ given by \eqref{gD}; thus $\gD$ is real.
  \begin{romenumerate}
	\item
If $\gD>0$, then $f$ has three simple real roots.
\item
If $\gD<0$, then $f$ has one simple real root, and two non-real complex
roots forming a pair of complex conjugates.
\item
If $\gD=0$, then $f$ has either one double and one simple real root, or a
triple real root.
  \end{romenumerate}
\end{theorem}
\begin{proof}
  Let, as in \refS{S3}, the 3 roots of $f$ in $\bbC$ be $\ga_1$, $\ga_2$,
  $\ga_3$. Note that $\ga_1+\ga_2+\ga_3=-b$. 
$\gD$ is real by \eqref{gD}. By \eqref{gD0}, $\gD=0$ if and only if two of
  the three roots coincide, \ie, $f$ has a double or triple root. This root
  has to be real, since otherwise its conjugate would be another double or
  triple root and $f$ would have at least 4 roots (counted with
  multiplicity), which is impossible.
If there is a real double root $\ga_1=\ga_2$, then $\ga_3=-b-\ga_1-\ga_2$ is
  real too.
This shows (iii).

Now suppose that $\gD\neq0$; thus $f$ has three distinct simple roots
in $\bbC$.
Since $f(x)\to-\infty$ as $x\to-\infty$ and $f(x)\to\infty$ as $x\to\infty$,
$f$ has at least one real root by the intermediate value theorem.
Further, since the roots are simple, $f$
changes sign at each root, so $f$ has an odd number of real roots. Hence $f$
has either 1 or 3 real roots.

If $f$ has 3 distinct real roots  $\ga_1$, $\ga_2$, $\ga_3$, then \eqref{gD0}
yields $\gD>0$.

If $f$ has only one real root, then the roots are $\ga_1$, $\ga_2$,
$\overline{\ga_2}$ with $\ga_1\in\bbR$ and $\ga_2\notin\bbR$, and
\eqref{gD0} yields
\begin{equation*}
  \gD=(\ga_1-\ga_2)^2(\ga_1-\overline{\ga_2})^2(\ga_2-\overline{\ga_2})^2
=|\ga_1-\ga_2|^4(2\ii\Im\ga_2)^2
<0.
\qedhere
\end{equation*}
\end{proof}

\begin{remark}
  \label{RRg}
For a cubic $y^3+py+q$ without second degree term,
$\gD/108=-(p/3)^3-(q/2)^2$ by \eqref{gD}, and thus the criterion for case
(i) is $(p/3)^3+(q/2)^2<0$; equivalently, $p<0$ and $|p/3|^3>|q/2|^2$.
This was found already by Cardano, see \refR{Rci}.
\end{remark}

The number of real roots is thus easily found. Now let us consider finding
the root(s). There are by \refT{TR3} three cases, which we treat
separately since they turn out to be quite different.
(Of course, the roots are always given by \refT{T3}, but we now want to
perform only real arithmetic, if possible.)

\subsection{$\gD=0$, a double or triple root}\label{SSAD0}

In the case $\gD=0$, the roots are easily found. The double (or triple) root
$\ga_1=\ga_2$ is also a root of the quadratic equation $f'(x)=0$
(choosing the root that also satisfies $f(x)=0$), and then
$\ga_3$ is given by $\ga_3=-b-2\ga_1$. 
It is easily seen that
$\ga_1=\ga_2=-b/3+\gd$, and thus  $\ga_3=-b/3-2\gd$, where
$\gd=\pm\sqrt{-p/3}$ with the correct sign given by $\sign(\gd)=\sign(q)$.
(See also Remarks \refand{R30}{RAci}.)

\subsection{$\gD<0$, one simple real root}

If $\gD<0$, then Cardano's formula \eqref{t3b} yields the unique real root
of $f$ by choosing the real cube roots. 
Note that $-q/2\pm\sqrt{-\gD/108}$
is real and that the product of the real cube roots in \eqref{t3b} is $-p/3$
as required, because the product is a cube root of $(-p/3)^3$ and $-p/3$ is
real.

\subsection{$\gD>0$, three simple real roots:
\ci}

The case $\gD>0$ is much more complicated.
Of course, Cardano's formula \eqref{t3b} still applies, but now
$\sqrt{-\gD/108}$ is imaginary 
and $-q/2\pm\sqrt{-\gD/108}$ complex,
so the
formula necessarily involves taking cube roots 
of complex (nonreal) numbers, even though we know that the final answer is a
real root of $f$; the three different choices of cube roots 
of $-q/2+\sqrt{-\gD/108}$
lead to the three different
real roots of $f(x)=0$. 

In this case the imaginary parts thus cancel in \eqref{t3b} for any
admissible
choice of cube roots.
This can also be seen as follows:
since $-q/2+\sqrt{-\gD/108}$ and $-q/2-\sqrt{-\gD/108}$ are complex
conjugates, we may, and have to, choose
cube roots of them that are complex conjugates in \eqref{t3b}; recall that
the product of these cube roots has to be $-p/3$, which is real. Hence,
Cardano's formula \eqref{t3b} for the roots may be written 
\begin{equation}\label{ci0}
   -\frac b3 +
  \sqrt[3]{-\frac q2+\sqrt{\frac{-\gD}{108}}}
+
\overline{\sqrt[3]{-\frac q2+\sqrt{\frac{-\gD}{108}}}}
=
   -\frac b3 +
2\Re  \sqrt[3]{-\frac q2+\sqrt{\frac{-\gD}{108}}}.
\end{equation}

Every complex number may be represented by a pair of real numbers, \viz{}
its real and imaginary parts, but taking the cube root of a complex number
may \emph{not} be reduced to a combination of real cube roots (or real
square and higher
roots) and usual algebraic algebraic operations.
In fact, it can be shown by Galois theory that 
if $f$ is any polynomial with rational coefficients (\ie, $f(x)\in\bbQ[x]$)
such that $f$ is irreducible over $\bbQ$ and has positive discriminant,
and $\ga$ is a root of $f$, then $\ga$ cannot be expressed
by real radicals; 
in other words, there does not exist a sequence of field extensions 
$\bbQ=F_0\subset F_1\subset F_2\dots \subset F_N$ where
$F_k=F_{k-1}[u_k]$ for some real $u_k$ with $u_k^{n_k}\in F_{k-1}$ for some
positive integer $n_k$, $k=1,\dots,N$, and $\ga\in F_N$;
see  \cite[Section 8.8]{vdW1} or \cite[Section 8.6]{Cox}. 
(Here $\bbQ$ may be replaced by any subfield of $\bbR$.) 

The case $\gD>0$ is known as the 
\emph{casus irreducibilis}; in this case, thus the equation
cannot (in general) be solved by radicals using only real numbers
(somewhat paradoxically, since the answers all are real).
This case is, by \eqref{gD}, characterized by $-4p^3>27q^2$, or equivalently
\begin{align}\label{ci1}
  p<0 \quad\text{and}\quad 4|p|^3>27q^2\ge0
\end{align}
or
\begin{equation}\label{ci2}
    p<0 \quad\text{and}\quad |p/3|^3>(q/2)^2.
\end{equation}

An alternative to Cardano's formula \eqref{ci0} in the case $\gD>0$ are the
following trigonometric formulas, which involves only real numbers but use
transcendental functions instead of algebraic expressions.

\begin{theorem}
  \label{TTrig}
If the real cubic polynomial $f(x)=x^3+bx^2+cx+d$ has positive discriminant
$\gD>0$, 
or equivalently, \eqref{ci2} holds,
then $f$ has three real roots given by
\begin{equation}
  \label{cos}
-\frac{b}3+2\sqrt{\frac{-p}3}
 \cos\lrpar{\frac13\arccos\lrpar{\frac{-q/2}{(-p/3)\qqc}}}
\end{equation}
where different branches of $\arccos$ yield the three different roots.

Equivalently, the roots are given by the three different values of
\begin{equation}
  \label{sin}
-\frac{b}3+2\sqrt{\frac{-p}3}
 \sin\lrpar{\frac13\arcsin\lrpar{\frac{q/2}{(-p/3)\qqc}}}.
\end{equation}
\end{theorem}

\begin{proof}
  Let $z\=u^3=-q/2+\sqrt{-\gD/108}=-q/2+\ii\sqrt{\gD/108}$, see \eqref{u3}.
Then, using \eqref{gD},
  \begin{equation*}
	|z|^2=\frac{q^2}{4}+\frac{\gD}{108}
=-\frac{p^3}{27}=\parfrac{-p}3^3
=\left|\frac p3\right|^3.
  \end{equation*}
We write $z$ in polar coordinates: $z=re^{\ii\gf}$, where thus
$r=|z|=(|p|/3)\qqc$ and 
\begin{equation*}
  \cos\gf=\frac{\Re z}{|z|}=\frac{-q/2}{r}
=\frac{-q/2}{(-p/3)\qqc}.
\end{equation*}
By \eqref{ci0}, a root of $f$ is given by
\begin{equation*}
  -\frac b3 + 2 \Re z\qqq
=
  -\frac b3 + 2 \Re r\qqq e^{\ii \gf/3}
=
  -\frac b3 + 2 r\qqq \cos\xpar{ \gf/3},
\end{equation*}
which yields \eqref{cos}, with different choices of $\gf$ yielding the
three roots.

To see \eqref{sin}, let $\psi\=\gf+3\pi/2$ and note that
$\sin\psi=-\cos\gf$ and $\sin(\psi/3)=\cos(\gf/3)$.
\end{proof}

In \eqref{cos} and \eqref{sin} we find the different roots by choosing
different values of $\arccos$ or $\arcsin$. Often it is more convenient to
make a single choice (for example the principal value with $0\le\gf\le\pi$ or
$-\pi/2\le\psi\le\pi/2$, but any choice will do).

\begin{theorem}
  \label{TTrig2}
Suppose that the real cubic polynomial $f(x)=x^3+bx^2+cx+d$ has positive
discriminant $\gD>0$,
or equivalently, that \eqref{ci2} holds.
Then, for any choice of
\begin{equation}\label{gf}
  \gf\=\arccos\lrpar{\frac{-q/2}{(-p/3)\qqc}},
\end{equation}
$f$ has three real roots 
\begin{equation}
  \label{cos2}
-\frac{b}3+2\sqrt{\frac{-p}3} \cos{\frac\gf3},
\qquad
-\frac{b}3-\sqrt{\frac{-p}3}
  \lrpar{ \cos{\frac\gf3}\pm\sqrt3\, \sin{\frac\gf3}}.
\end{equation}
Similarly, for any choice of
\begin{equation}\label{psi}
  \psi\=\arcsin\lrpar{\frac{q/2}{(-p/3)\qqc}},
\end{equation}
$f$ has three real roots 
\begin{equation}
  \label{sin2}
-\frac{b}3+2\sqrt{\frac{-p}3} \sin{\frac\psi3},
\qquad
-\frac{b}3-\sqrt{\frac{-p}3}
  \lrpar{ \sin{\frac\psi3}\pm\sqrt3\, \cos{\frac\psi3}}.
\end{equation}
\end{theorem}

\begin{proof}
The other possible values of $\arccos$ in \eqref{gf} are
$\pm\gf+2\pi n$, $n\in\bbZ$. Hence, the three values of the $\cos$ in
\eqref{cos} are $\cos(\gf/3)$ and 
\begin{equation*}
\cos\frac{\gf\pm2\pi}3
=\cos\frac\gf3\cos\frac{2\pi}3\mp\sin\frac\gf3\sin\frac{2\pi}3
=-\frac12\cos\frac\gf3 \mp\frac{\sqrt3}2\sin\frac\gf3  ,
\end{equation*}
and \eqref{cos} yields \eqref{cos2}.  

Similarly,
the other possible values of $\arcsin$ in \eqref{psi} are
$\psi+2\pi n$ and $3\pi-\psi+2\pi n$, $n\in\bbZ$, and
the three values of the $\sin$ in
\eqref{sin} are $\sin(\psi/3)$ and 
\begin{equation*}
\sin\frac{\psi\pm2\pi}3
=\sin\frac\psi3\cos\frac{2\pi}3\pm\cos\frac\psi3\sin\frac{2\pi}3
=-\frac12\sin\frac\psi3 \pm\frac{\sqrt3}2\cos\frac\psi3.
\qedhere
\end{equation*}
\end{proof}

\begin{remark}\label{RAci}
  The formulas \eqref{cos}--\eqref{sin2} are
 meaningful (with real quantities only),
  exactly when $p<0$ and $|q/2|\le |p/3|\qqc$, \ie, when \eqref{ci2} holds 
or in the limiting case $\gD=0$ and $p\neq0$. 
The formulas \eqref{cos}--\eqref{sin2} are
  valid in the latter case too, and then yield the roots 
$-b/3+\gd$, $-b/3+\gd$, $-b/3-2\gd$, where
$\gd=\pm\sqrt{-p/3}$ with $\sign(\gd)=\sign(q)$,
as found more easily in \refSS{SSAD0}.
\end{remark}

\begin{example}\label{Eci1}
Let $f(x)=x^3-x$, which evidently has the three real roots $0$, $\pm1$.

We have $b=0$, $p=c=-1$, $q=d=0$, and, by \eqref{gD}, $\gD=4$ (which is
verified by \eqref{gD0}).
Hence, 
$u^3=-q/2+\sqrt{-\gD/108}=\sqrt{-1/27}=3^{-3/2}\ii$,
and we find the three cube roots
\begin{align*}
  u_1&=\frac{-\ii}{\sqrt3},
\\
u_2&=\frac{-\ii}{\sqrt{3}}\Bigpar{-\frac12+\ii\frac{\sqrt3}{2}}
=\frac12+\frac{\ii}{2\sqrt3},
\\
u_3&=\frac{-\ii}{\sqrt{3}}\Bigpar{-\frac12-\ii\frac{\sqrt3}{2}}
=-\frac12+\frac{\ii}{2\sqrt3}.
\end{align*}
Hence \eqref{t3a} and \eqref{t3b} yield the three roots of $f$ as
\begin{align*}
  u_1+\overline{u_1}&=0, &
  u_2+\overline{u_2}&=1, &
  u_3+\overline{u_3}&=-1.
\end{align*}

Alternatively, we may use the trigonometric formula \eqref{cos}.
We have $\arccos\bigpar{-(q/2)\big/(-p/3)^{3/2}}=\arccos0=\pi/2+n\pi$,
$n\in\bbZ$, and thus the
three roots are
\begin{align*}
  \frac{2}{\sqrt3}\cos\frac{\pi}{6}&=1, &
  \frac{2}{\sqrt3}\cos\frac{5\pi}{6}&=-1, &
  \frac{2}{\sqrt3}\cos\frac{9\pi}{6}&=0.
\end{align*}

Similarly, \eqref{sin} yields,
since $\arcsin\bigpar{(q/2)\big/(-p/3)^{3/2}}=\arcsin0=n\pi$, $n\in\bbZ$,
the three roots as
\begin{align*}
  \frac{2}{\sqrt3}\sin0&=0, &
  \frac{2}{\sqrt3}\sin\frac{2\pi}{3}&=1, &
  \frac{2}{\sqrt3}\sin\frac{4\pi}{3}&=-1.
\end{align*}

Using \eqref{cos2} with $\gf=\pi/2$ we find the same roots given as
\begin{equation*}
    \frac{2}{\sqrt3}\cos\frac{\pi}{6},
\qquad
-  \frac{1}{\sqrt3}\lrpar{\cos\frac{\pi}{6}\pm \sqrt3\sin\frac{\pi}{6}}
\end{equation*}
while \eqref{sin2} with $\psi=0$ yields
\begin{equation*}
    \frac{2}{\sqrt3}\sin0=0,
\qquad
-  \frac{1}{\sqrt3}\lrpar{\sin0\pm \sqrt3\cos0}=\mp1.
\end{equation*}
\end{example}

\begin{example}
  \label{Eci2}
Let $f(x)=x^3-7x-6$, which has the roots $3$, $-1$, $-2$.
We have $p=-7$, $q=-6$ and $\gD=400$.
Thus, 
$u^3=-q/2+\sqrt{-\gD/108}=3+\sqrt{-100/27}=3+\frac{10}{3^{3/2}}\ii$,
and we find the three cube roots
\begin{align*}
  u_1&=-1+\frac{2}{\sqrt3}\ii,
&
u_2&=-\frac12 -\frac{5}{2\sqrt{3}}\ii,
&
u_3&=\frac32 +\frac{1}{2\sqrt{3}}\ii.
\end{align*}
Hence, \eqref{t3a} and \eqref{t3b} yield the three roots of $f$ as
\begin{align*}
 2\Re u_1&=-2, &
 2\Re u_2&=-1, &
2\Re  u_3&=3.
\end{align*}
The trigonometric solution \eqref{cos} yields
\begin{equation}
  2\sqrt{\frac 73}
  \cos\lrpar{\frac13\arccos\sqrt{\frac{243}{343}}+\frac{2\pi k}3},
\qquad k=0,1,2,
\end{equation}
and it is far from obvious that this yields three integers $3$, $-2$, $-1$.
\end{example}

\begin{example}
  \label{Eci3}
Let $f(x)=x^3-7x^2+14x-8$, which has the roots 1, 2, 4.
Then, by \eqref{p}--\eqref{gD}, $p=-7/3$, $q=-20/27$ and $\gD=36$.
Thus,
$u^3=-q/2+\sqrt{-\gD/108}=\frac{10}{27}+\frac{\ii}{\sqrt 3}$,
and we find the three cube roots
\begin{align*}
  u_1&=-\frac23+\frac{1}{\sqrt3}\ii,
&
u_2&=-\frac16 -\frac{\sqrt3}{2}\ii,
&
u_3&=\frac56 +\frac{1}{2\sqrt{3}}\ii.
\end{align*}
Hence,  \eqref{t3b} yields the three roots of $f$ as
\begin{align*}
\frac73+ 2\Re u_1&=1, &
\frac73+ 2\Re u_2&=2, &
\frac73+2\Re  u_3&=4.
\end{align*}
The trigonometric solution \eqref{cos} yields
\begin{equation}
\frac73+ \frac{ 2\sqrt{7}}{3}
  \cos\lrpar{\frac13\arccos\frac{10}{7\sqrt7}+\frac{2\pi k}3},
\qquad k=0,1,2,
\end{equation}
which, again surprisingly, yields three integers $4$, $1$, 2.
\end{example}

\begin{example}
  \label{Eci4}
Cardano saw the problem with the \ci{} and asked Tartaglia about it, giving
$x^3=9x+10$ as an example, see \refR{Rci}.
In this case, thus $f(x)=x^3-9x-10$, so $p=-9$, $q=-10$ and 
$\gD/108=(-p/3)^3-(q/2)^2=2$.
Thus,
$u^3=-q/2+\sqrt{-\gD/108}=5+{\sqrt 2}\ii$,
and we find the three cube roots
\begin{align*}
  u_1&=-1+{\sqrt2}\,\ii,
&
u_2&=\frac{1+\sqrt6}2+\frac{\sqrt3-\sqrt2}{2}\ii,
&
u_3&=\frac{1-\sqrt6}2-\frac{\sqrt3+\sqrt2}{2}\ii.
\end{align*}
Hence,  \eqref{t3a} and \eqref{t3b} yield the three roots of $f$ as
\begin{align*}
 2\Re u_1&=-2, &
 2\Re u_2&=1+\sqrt6, &
2\Re  u_3&=1-\sqrt6.
\end{align*}
\end{example}

\begin{example}
  \label{Eci5}
Cardano \cite[Chapter XIII]{ArsMagna} considered also the 
equation $y^3=8y+3$.
He saw that $y=3$ is one solution (without discussing the problem of the
\ci, see \refR{Rci}).  
In modern terms he then found the other
two solutions by finding the roots of the quadratic polynomial
$(y^3-8y-3)/(y-3)=y^2+3y+1$; he gave a general formula for this.
(The other two solutions are $-(3\pm\sqrt5)/2$; these
are negative, and Cardano changes the sign and
interprets the result $(3\pm\sqrt5)/2$ as the two positive solutions of
$x^3+3=8x$.)

Let us instead use Cardano's formula.
In this case, 
$p=-8$, $q=-3$ and $\gD=-4p^3-27q^2=1805$.
Thus, 
\begin{equation}
 u^3=-q/2+\sqrt{-\gD/108}=\frac32+\frac{19\sqrt{5}}{6\sqrt 3}\ii, 
\end{equation}
and we find the three cube roots
\begin{align*}
  u_1&=\frac32+\frac{\sqrt5}{2\sqrt3}\ii,
&
u_2&=\frac{\sqrt5-3}4-\frac{9+\sqrt5}{4\sqrt3}\ii,
&
u_3&=-\frac{\sqrt5+3}4+\frac{9-\sqrt5}{4\sqrt3}\ii.
\end{align*}
Hence,  \eqref{t3a} and \eqref{t3b} yield the three solutions of $y^3=8y+3$ as
\begin{align*}
 2\Re u_1&=3, &
 2\Re u_2&=-(3-\sqrt5)/2, &
2\Re  u_3&=-(3+\sqrt5)/2.
\end{align*}
\end{example}

\begin{example}
  \label{Eci6}
\label{Ecilast}
Bombielli (1550) 
considered the equation $y^3-15y-4=0$.
Cardano's formula \eqref{t3a} yields the roots as
\begin{equation}\label{bomb}
\sqrt[3]{2+\sqrt{-121}}+\sqrt[3]{2-\sqrt{-121}}
= \sqrt[3]{2+11\ii}+\sqrt[3]{2-11\ii}.
\end{equation}
Bombielli noted that 4 is a root, and showed in a pioneering calculation
with complex numbers that
$(2\pm\ii)^3=2\pm11\ii$, and thus \eqref{bomb} correctly yields the root
$(2+\ii)+(2-\ii)=4$.

The two other cube roots of $2+11\,\ii$ are
\begin{align*}
u_2&=-1-\frac{\sqrt3}2+\frac{2\sqrt3-1}{2}\ii,
&
u_3&=-1+\frac{\sqrt3}2+\frac{-2\sqrt3-1}{2}\ii.
\end{align*}
Hence,  the three solutions of $y^3-15y-4=0$ are
$4$ and $-2\pm\sqrt 3$.
\end{example}

\subsection{Further comments for real coefficients}

\begin{remark}
  In the \emph{casus irreducibilis} we thus always obtain the roots as
  complicated expression involving complex cube roots, even when the roots
  are, for example, simple integers. (See Examples \ref{Eci1}--\ref{Ecilast} for
  simple cases.) 

Also in the case of a single real root, Cardano's formula typically yields
complicated expressions (but now involving real roots only)
also for simple integer solutions.
\end{remark}

\begin{example}
  The equation $x^3+6x=20$ \cite[Chapter XI]{ArsMagna} has the real root
  $x=2$ (and the complex roots $-1\pm3\ii$).
Cardano's formula yields the root as
\begin{equation}\label{ebad2}
  \sqrt[3]{10+\sqrt{108}} +   \sqrt[3]{10-\sqrt{108}}
=  \sqrt[3]{10+\sqrt{108}} -   \sqrt[3]{\sqrt{108}-10}.
\end{equation}
This indeed equals 2 because $\sqrt{108}\pm10=(\sqrt 3\pm1)^3$, but this is
far from obvious and it is hard to see how \eqref{ebad2} can be simplified
without knowing the answer.
\end{example}

\begin{remark}
The trigonometric formulas \eqref{cos}--\eqref{sin2} are valid also when
$\gD<0$ (and, more generally, for arbitrary complex coefficients with
$p\neq0$), but then the angles $\gf$ and $\psi$ are complex and the formulas
less useful.

For real coefficients with $\gD<0$ (the case of a single real root), we can
choose $\gf$ in \eqref{gf} or $\psi$ in \eqref{psi}  purely imaginary
(after a change of sign of the roots and $q$ in the case $p<0<q$), and the
formulas \eqref{cos2} and \eqref{sin2} can be rewritten with real hyperbolic
functions as follows 
\cite{Holmes}. 
\end{remark}

\begin{theorem}
If the real cubic polynomial $f(x)=x^3+bx^2+cx+d$ has negative discriminant
$\gD<0$, then $f$ has one real and two conjugate complex roots given as follows.
\begin{romenumerate}
\item
If $p<0$ and $q/2<-|p/3|\qqc$, then the roots are
\begin{align*}
-\frac{b}3&+2\sqrt{\frac{-p}3}
 \cosh\lrpar{\frac13\arccosh\lrpar{\frac{-q/2}{(-p/3)\qqc}}},
\\
-\frac{b}3&-\sqrt{\frac{-p}3}
  \lrpar{ \cosh\lrpar{\frac13\arccosh\lrpar{\frac{-q/2}{(-p/3)\qqc}}}
\pm\sqrt3\,\ii \sinh\lrpar{\frac13\arccosh\lrpar{\frac{-q/2}{(-p/3)\qqc}}} }.
\end{align*}
\item
If $p<0$ and $q/2>|p/3|\qqc$, then the roots are
\begin{align*}
-\frac{b}3&-2\sqrt{\frac{-p}3}
 \cosh\lrpar{\frac13\arccosh\lrpar{\frac{q/2}{(-p/3)\qqc}}},
\\
-\frac{b}3&+\sqrt{\frac{-p}3}
  \lrpar{ \cosh\lrpar{\frac13\arccosh\lrpar{\frac{q/2}{(-p/3)\qqc}}}
\pm\sqrt3\,\ii \sinh\lrpar{\frac13\arccosh\lrpar{\frac{q/2}{(-p/3)\qqc}}} }.
\end{align*}

\item
If $p>0$, then the roots are
\begin{align*}
-\frac{b}3&-2\sqrt{\frac{p}3}
 \sinh\lrpar{\frac13\arcsinh\lrpar{\frac{q/2}{(p/3)\qqc}}},
\\
-\frac{b}3&+\sqrt{\frac{p}3}
  \lrpar{ \sinh\lrpar{\frac13\arcsinh\lrpar{\frac{q/2}{(p/3)\qqc}}}
\pm\sqrt3\,\ii \cosh\lrpar{\frac13\arcsinh\lrpar{\frac{q/2}{(p/3)\qqc}}} }.
\end{align*}
\end{romenumerate}
\end{theorem}

\begin{remark}\label{Rnickalls}
The three cases in \refT{TR3} can also be seen geometrically by considering
the graph of $f$ (or $g$) as follows (based on \cite{Nickalls:cubic}).

Denote the stationary points of $g$, \ie{} the roots of $g'(y)=0$, by
$\pm\gd$; thus the stationary points of $f$ are 
$-\frac b3\pm\gd$. 
(Note also that $f$ has an inflection point at $(-\frac b3,q)$.)
Since $g'(y)=3y^2+p$, we have
\begin{equation}\label{delta}
  \gd=\sqrt{-p/3}.
\end{equation}
In particular, $\gd$ is either real ($p\le0$) or imaginary ($p>0$); when
$p<0$, we choose for convenience the positive square root in \eqref{gd}.
Let further 
\begin{equation}\label{h}
 h\=-\bigpar{g(\gd)-g(0)}=-\bigpar{\gd^3+p\gd}=2\gd^3.
\end{equation}
(We choose this sign so that $h>0$ when $\gd>0$.) Thus
\begin{equation}\label{erika}
  f(-\tfrac b3\pm\gd)=g(\pm\gd)=g(0)\mp h=q\mp h.
\end{equation}

If $\gd>0$, 
then $f$ thus has a local maximum at $-\frac b3-\gd$ with value $q+h$,
and a local minimum at $-\frac b3+\gd$ with value $q-h$.
Considering the graph of $f$, we see that $f(x)=0$ then has
three real roots if $0\in(q-h,q+h)$,
two real roots (of which one double) if $0=q\pm h$, and one (simple) real
root if $0\notin[q-h,q+h]$. We thus see the three different cases in
\refT{TR3}, with (i) if $h>|q|$, \ie{} $h^2>q^2$,
(ii) if $h<|q|$, \ie{} $h^2<q^2$, and
(iii) if $h=|q|$, \ie{} $h^2=q^2$.

In the limiting case $\gd=0$ (which entails $h=0$), $f$ has no local maximum
or minimum, but a saddle point at 
$-\frac b3$ with $f(-\frac b3)=g(0)=q$. In this case there is a triple root
(case (iii))
if $q=0$ and otherwise one simple real root (case (ii)).

If $\gd$ is imaginary (and non-zero), then $f'(x)\neq0$ for all real $x$,
and thus 
$f'(x)>0$
(since this certainly holds for large $x$); hence $f$ is strictly increasing
and $f(x)=0$ has a single, simple root for any $q$ (case (ii)). In this
case, $h$ is 
imaginary too, so $h^2<0\le q^2$.

We thus see that in all cases, \refT{TR3} holds with case 
(i) when $h^2-q^2>0$,
(ii) when $h^2-q^2<0$,
and (iii) when $h^2-q^2=0$.
This is also confirmed by a simple calculation showing that 
\begin{equation}
  \gD=-4p^3-27 q^2=108 \gd^6-27q^2 =27(h^2-q^2).
\end{equation}

Using parameters $\gd$ and $h$ also simplify the formulas above a little.
Since $h^2=4\gd^6=-4p^3/27$, \eqref{u3} becomes
\begin{equation}
  u^3=\tfrac12\bigpar{-q+\sqrt{q^2-h^2}}
\end{equation}
so Cardano's formula \eqref{t3b} for the roots of $f$ becomes
\begin{equation}
  -\frac b3 +
\sqrt[3]{\tfrac12\bigpar{-q+\sqrt{q^2-h^2}}}
+ \sqrt[3]{\tfrac12\bigpar{-q-\sqrt{q^2-h^2}}}.
\end{equation}
Similarly, in the \emph{casus irreducibilis}, which now 
is characterized by $h>|q|$,
\eqref{cos} and \eqref{sin} can be written
\begin{equation}
  -\frac b3 + 2\gd\cos\lrpar{\frac13\arccos\Bigpar{-\frac q h}}
=  -\frac b3 + 2\gd\sin\lrpar{\frac13\arcsin\Bigpar{\frac q h}}.
\end{equation}
\end{remark}

\begin{remark}
Consider the case of a real cubic $f(x)=ax^3+bx^2+cx+d$ with a single real
root ($\gD<0$). 
A geometric way to find the two complex roots from the graph of $f$ (on
$\bbR$) has been given by \eg{} \cite{IrwinW,Henriquez}:
Let $A$ be the intersection of the curve $y=f(x)$ and the $x$-axis (\ie, the
real root), and let $\ell$ be a tangent from $A$ to the curve. If the tangent
point has $x$-coordinate $x_0$, and the slope of the tangent is $k$, then
the complex roots are $x_0\pm\sqrt{k}\,\ii$.
\end{remark}

\section[History]{History%
\protect\footnote{This appendix is largely based on 
the Foreword (by Oystein Ore) and Preface (by T. Richard Witmer) to
the English translation of \ArsMagna,
the text itself of \ArsMagna,
Katscher \cite{MathDL:Tartaglia,Tartaglia},
\citet{vdW},
and 
\emph{The MacTutor History of Mathematics} \cite{MT} on the Internet 
(the articles
Quadratic, cubic and quartic equations;
Scipione del Ferro;
Nicolo Tartaglia;
Girolamo Cardano; 
Lodovico Ferrari;
Tartaglia versus Cardan); much more details can be found in these
references.
For the history after Cardano's \ArsMagna, see \citet{vdW}.}
}\label{Ahistory}

The solution to cubic equations was first found c.~1515 by Scipione
del Ferro (1465--1526) in Bologna, at least for some cases.
At this time, negative numbers were not used, nor was 0. 
Thus (in modern notation)
$x^3+px=q$, $x^3=px+q$ and $x^3+q=px$, with positive $p$ and $q$,
were regarded as different types
of equations.
(The third type was often ignored. We know today that it always has one
negative solution, which generally was disregarded, and either zero 
positive solutions or two (\ci); hence Cardano's formula will never yield a
positive solution using real roots only.
The negative solution was treated by Tartaglia and Cardano by, in modern
terms, changing the sign of  $x$, which transforms $x^3+q=px$ to $x^3=px+q$.)
Scipione del Ferro could solve the first type and possibly the second 
(researchers disagree).

\begin{remark}
There are 13 types of nontrivial cubic equations with positive coefficients:
$x^3+cx=d$, $x^3=cx+d$, $x^3+d=cx$, 
$x^3=bx^2+d$, $x^3+bx^2=d$, $x^3+d=bx^2$,
$x^3+bx^2+cx=d$, $x^3+cx=bx^2+d$, $x^3+bx^2=cx+d$, 
$x^3=bx^2+cx+d$, $x^3+d=bx^2+cx$, $x^3+cx+d=bx^2$, $x^3+bx^2+d=cx$.
These are, for example, discussed
separately  in Cardano's \emph{Ars Magna} \cite[Chapters XI--XXIII]{ArsMagna}.

Similarly, Cardano \cite[Chapter V]{ArsMagna} considers 
three different types of quadratic equations:
$x^2=bx+c$, $x^2+bx=c$, $x^2+c=bx$ (as did al-Khwarizmi c.~800,
while Brahmagupta in 628 used both zero and negative numbers and treated all
quadratic equations together), 
and many types of quartic equations, see \refR{R4} below.

Note that Cardano discusses negative numbers and negative \emph{solutions} 
(called ``false solutions'')
to equations 
\cite[in particular  Chapters I and  XXXVII]{ArsMagna};
however, he does not consider negative
\emph{coefficients} (at least not usually, although there are occasional
uses in a few examples, 
for example \cite[Chapter XXXIX, Problem  IX]{ArsMagna}). 

Cardano even makes a
pioneering tentative use of imaginary numbers and complex solutions 
\cite[Chapter  XXXVII, Rule II]{ArsMagna}, although he clearly does not
understand 
them and he seems sceptical to his calculation.
Complex numbers were introduced in a consistent way somewhat
later by
Rafael Bombelli (1526--1572) 
in his book \emph{Algebra} (1572), where he also shows how to work with
negative numbers \cite[Rafael Bombielli]{MT}.
\end{remark}

\begin{remark}\label{Rreduce}
  It is claimed in \cite[Scipione del Ferro]{MT} that the reduction \eqref{fg} 
to an equation 
without quadratic term (which seems trivial to
us) was known at the time of del Ferro, but 
this seems incorrect, and I rather believe 
the claim by \cite{MathDL:Tartaglia} that del Ferro considered only
such cubics because the others were too difficult to be solved. 
See further \refR{Rred2}. 
\end{remark}

However, del Ferro kept his solution secret. The traditional story is that
he did not tell anyone about it until his deathbed in 1526, when
he told the solution to his student Antonio Maria Fior. (This seems a bit
exaggerated,  since 
his son-in-law Hannibal della Nave much later showed Cardano
a notebook written by del Ferro presenting the solution, but he certainly
told very few.)

Fior let it become known that
he could solve cubic equations (without disclosing the method).
This prompted Nicolo Tartaglia (1500--1557) in Venice to find solutions.
He first found a solution to some equations of the type $x^3+bx=d$. 
He claims
\cite[XIIII p.~12, XXV p.~15, p.~64]{Tartaglia} that he found the solution
to all such equations in 1530, 
but he really could solve (and construct) only special cases, in modern
terms having a negative integer solution.
A public contest was held between Fior and
Tartaglia in 1535,
where each was to solve 30 problems set by the other (within 40 or 50 days);
according to himself \cite[XXV p.~13]{Tartaglia},
Tartaglia managed to find the solutions to the two types 
$x^3+px=q$ and $x^3=px+q$ on 12 and 13 February 1535, only 8 days before the
deadline of the contest%
\footnote{According to  \cite[XXXI p.~29]{Tartaglia}, the
  contest was on 22 February, which yields a  discrepancy in the exact
  dates.}, 
and then
Tartaglia easily won by solving all 30 problems in 2 hours.
(Fior's problems, which are given in \cite[XXXI pp.~29--31]{Tartaglia}, 
were all of the type $x^3+px=q$, which he did not
believe that Tartaglia could solve.)

Girolamo Cardano (1501--1576) in Milan then invited Tartaglia, and
managed to make him disclose the method (25 March 1539), after Cardano had
promised 
Tartaglia to keep it secret until Tartaglia had published the method
himself (something Tartaglia never did, preferring to keep it secret
and regretting that he had told Cardano). Cardano worked on the
solution together with his young assistant Lodovico Ferrari (1522--1565),
who in 1541 found a solution to quartic equations. 

Cardano found out that the cubic equation had been solved by del Ferro
before Tartaglia, and used this as an excuse to break his promise to
Tartaglia and publish (in 1545) the solutions of cubic and quartic equations in
his large algebra book \emph{Ars Magna} \cite{ArsMagna}, where they form a
major part. (All 13 types of cubic equations are discussed separately
in detail, but
only some of the possible quartic equations, see below.)
Cardano introduces the solution of the cubic equation with:
\begin{quotation}
Scipio Ferro of Bologna well-nigh thirty years ago discovered this rule and
handed it on to Antonio Maria Fior of Venice, whose contest with Niccol\`o
Tartaglia of Brescia gave Niccol\`o occasion to discover it. He [Tartaglia]
gave it to
me in reponse to my entreaties, though withholding the demonstration. Armed
with this assistance, I sought out its demonstration in [various]
forms. This was very difficult. My version of it follows.
\cite[Chapter XI]{ArsMagna}
\end{quotation}

The publication led  to a bitter dispute between Tartaglia and
Cardano--Ferrari.
Tartaglia accused in a book \cite[XXXIIII p.~42]{Tartaglia} (1546)
Cardano of
breaking an oath to him to keep the solution secret; he also added some
insults against Cardano. 
This led to a series of equally insulting pamphlets (6 each)
by Ferrari (defending Cardano, who kept a low profile in the dispute)
and Tartaglia (renewing his accusations and insults), 
and finally to
a public contest between Tartaglia and Ferrari in Milan on 10 August
1548. (Each posed 62 problems to the other.
Ferrari won clearly; Tartaglia left Milan after the first day of the
contest, when he saw that he was losing.)

\begin{remark}\label{Rred2}
del Ferro, Fior and Tartaglia (with the exception $x^3+bx^2=d$ discussed above) 
considered only cubics without second degree term, see \refR{Rreduce},
It seems that the reduction \eqref{fg} of general cubic equations to this
case is due to Cardano, who in
 \cite{ArsMagna} uses this reduction in 9 of the 10 types with a
quadratic term (the exception is $x^3+d=bx^2$, which is reduced by the
substitution $x=d^{2/3}/y$).
(Cardano claims in the beginning of
\cite{ArsMagna} that those things to which he has not attached any name are
his own discoveries. This is of course  no proof that this reduction is his
own invention, but it suggests that he
regarded the reduction either as his own contribution or trivial.)
Note that
Cardano does the reduction separately for each type and that he
does not discuss the reduction in his earlier chapters on some
transformatons of equations.
Moreover, he surprisingly does \emph{not} use the corresponding reduction
for fourth degree 
equations
(see \refR{R4}).
Furthermore,  Tartaglia did not know this reduction (until he
read \cite{ArsMagna}); note
that Tartaglia himself only mentions cubics without second degree term in
the poem that he later claimed that he gave Cardano with the solution
(see \refR{Rtart}), and that when he claims to have solved $x^3+bx^2=d$ in 1530,
he says that he had not been able to solve $x^3+bx^2+cx=d$ 
\cite[XIIII  p.~12]{Tartaglia} (and there is no indication that he found a
solution later).
\end{remark}

\begin{remark}
  \label{Rci}
Cardano quickly realized the problem with the \emph{casus irreducibilis},
see \refApp{Areal},
and wrote to Tartaglia about it on 4 August 1539
\cite[XXXVIII  p.~48]{Tartaglia},
giving the correct
condition for it (see \refR{RRg}) and giving $x^3=9x+10$ as an example
(see \refE{Eci4}).
Tartaglia was no longer cooperative, but
it seems that neither Cardano nor Tartaglia understood how to handle this
case.

Cardano ignores the complications of the \ci{} in \ArsMagna{}. 
In \cite[Chapter
  XIII]{ArsMagna} he solves $y^3=8y+3$, and claims that he obtains $y=3$
(which clearly is a solution) by his method, which seems to be at best an
oversimplification. (Cf.{} \refE{Eci5}.)
\end{remark}

\begin{remark}\label{R4}
Cardano  lists \cite[Chapter XXXIX]{ArsMagna}
20 types of quartic equations that he can solve;
these are the 10 nontrivial cases without cubic term (excluding the ones with
only even powers of $x$, which are quadratic equations in $x^2$) and,
symmetrically,  
the 10 nontrivial cases without linear term
(which are reduced to the former by inversion).

Cardano  states that these cases ``are the most general as there are 67
others''; I do not 
understand which these 67 other cases are. 
Moreover, 
there are 15 cases with all possible terms
(cubic, quadratic, linear and constant), and 7 additional without quadratic
terms; these are not mentioned as far as I can see.

Cardano gives several examples where quartic equations are solved by
Ferrari's method (see \refApp{A4}); these examples illustrate
4 of the 10 types without cubic term and 2 of the 10 types without linear
term, and it is clear that the method applies
to all 20 types.

There is also a single example of an equation with both linear and cubic
terms ($x^4+2x^3=x+1$, \cite[Problem XXXIX.XIII]{ArsMagna}), 
but this is solved by special argument reducing this
equation to a succession of two quadratic equations
(the equation implies $(x(x+1))^2=x(x+1)+1$ so $x(x+1)$ is the golden ratio
$(\sqrt5+1)/2$).

Note that Cardano \cite{ArsMagna}
does \emph{not} use the general reduction \eqref{fg4} to
eliminate the cubic term (in analogy with his treatment of cubic equations),
which, together with Ferrari's method, would have given the solution of all
types of quartic equations. I do not know whether this reduction, and thus
the solution to general quartics,  was
found by Cardano, Ferrari or someone else.
\end{remark}

\section{del Ferro's solution of the cubic equation}\label{A3}

Of course, del Ferro, Tartaglia and Cardano did not know Galois theory when
they found the solution in \refT{T3}.
Their method is more direct, and consists in observing (by a stroke of genius)
that if $y=u+v$, then
\begin{equation}
  y^3=(u+v)^3=u^3+v^3+3uv(u+v)=u^3+v^3+3uvy;
\end{equation}
hence, if we can find two numbers $u$ and $v$ such that
\begin{align}
  u^3+v^3&=-q \label{df+}\\
3uv&=-p,\label{df}
\end{align}
then $y^3=-py-q$, so $y$ is a root of $g(y)=0$.
Note that \eqref{df+}--\eqref{df} are the same as \eqref{u3+v3} and
\eqref{uv}.
To find $u$ and $v$, we multiply \eqref{df+} by $u^3$ and substitute
\eqref{df}, yielding
\begin{equation}\label{u6}
  u^6+qu^3+(-p/3)^3=0.
\end{equation}
This is a quadratic equation in $u^3$, which is readily solved and yields
\eqref{u3}; then $u$ is found by taking the cube root and $v$ is found from
\eqref{df}. We see that this yields the same $u$ and
$v$ as the argument in \refS{S3}. (In particular, \eqref{v3} holds, which
shows that choosing the other root in \eqref{u6} just means interchanging
$u$ and $v$, which does not change the root $u+v$; this should be no surprise,
since $u$ and $v$ have identical roles in the \emph{ansatz} $y=u+v$.) 
Consequently, 
this straightforward method yields the same solution $u+v$ as given
in \eqref{gb1} and \eqref{t3a}, and we obtain another proof of \refT{T3}.
(To see that the three different choices of $u$ as a cube root of $u^3$
really yield the three different roots of $g(y)=0$, with correct
multiplicities if there is a double root, is perhaps less obvious by this
method. We do not give a direct proof since we already know from \refS{S3}
that this indeed is the case.)

\begin{remark}\label{Rtart}
  Actually, the method just described, with $y=u+v$, is Tartaglia's and 
Cardano's (and possibly del Ferros's) version
  for the equation $y^3=cy+d$ (with $c,d>0$) \cite[Chapter XII]{ArsMagna}
which corresponds to our $p<0$, $q<0$. 
For the equation $y^3+cy=d$,
which corresponds to our $p>0$, $q<0$, del Ferro, Tartaglia and Cardano
instead set $y=u-v$
\cite[Chapter XI]{ArsMagna}, using
\begin{equation}
  y^3=(u-v)^3=u^3-v^3-3uv(u-v)=u^3-v^3-3uvy,
\end{equation}
and then find $u$ and $v$ such that $u^3-v^3=-q$ and $3uv=p$. This just
means changing the sign of $v$ in the equations above, which of course
yields the same final result. (But it keeps $u$ and $v$ positive in both
cases.) 

The third case without second degree term, $y^3+d=cy$ is reduced by Cardano
to the  case $y^3=cy+d$
\cite[Chapter XIII]{ArsMagna}, 
essentially by substituting $-y$ for $y$, although
  Cardano expresses this differently.
\end{remark}

According to Tartaglia \cite[XXXIIII  pp.~42--43]{Tartaglia}, 
he gave these rules 25 March 1539
to Cardano in form of the
following poem 
(English translation from \cite[Tartaglia versus Cardan]{MT}):

\begin{verse}
When the cube and things together\\
Are equal to some discreet number,\\
Find two other numbers differing in this one.\\
Then you will keep this as a habit\\
That their product should always be equal\\
Exactly to the cube of a third of the things.\\
The remainder then as a general rule\\
Of their cube roots subtracted\\
Will be equal to your principal thing\\
In the second of these acts,\\
When the cube remains alone,\\
You will observe these other agreements:\\
You will at once divide the number into two parts\\
So that the one times the other produces clearly\\
The cube of the third of the things exactly.\\
Then of these two parts, as a habitual rule,\\
You will take the cube roots added together,\\
And this sum will be your thought.\\
The third of these calculations of ours\\
Is solved with the second if you take good care,\\
As in their nature they are almost matched.\\
These things I found, and not with sluggish steps,\\
In the year one thousand five hundred, four and thirty.%
\footnote{Venice reckoned the year from 1 March, so February 1735 was still 
1734 in Venice
\cite{MathDL:Tartaglia}.}
\\
With foundations strong and sturdy\\
In the city girdled by the sea. 
\end{verse}

The Italian original (which rhymes in the form \emph{terza rima}) is
\cite{MathDL:Tartaglia}: 

\begin{verse}
\newcommand\xx{\vskip 6pt}
Quando chel cubo con le cose appresso\\
Se agguaglia à qualche numero discreto\\
Trouan dui altri differenti in esso.\\
\xx
Dapoi terrai questo per consueto\\
Che'llor produtto sempre sia eguale\\
Alterzo cubo delle cose neto,\\
\xx
El residuo poi suo generale\\
Delli lor lati cubi ben sottratti\\
Varra la tua cosa principale.\\
\xx
In el secondo de cotestiatti\\
Quando che'l cubo restasse lui solo\\
Tu osseruarai quest'altri contratti,\\
\xx
Del numer farai due tal part'\`a uolo\\
Che l'una in l'altra si produca schietto\\
El terzo cubo delle cose in stolo\\
\xx
Delle qual poi, per communprecetto\\
Torrai li lati cubi insieme gionti\\
Et cotal somma sara il tuo concetto.\\
\xx
El terzo poi de questi nostri conti\\
Se solue col secondo se ben guardi\\
Che per natura son quasi congionti.\\
\xx
Questi trouai, \& non con pa\ss i tardi\\
Nel mille cinquecent\`e, quatroe trenta\\
Con fondamenti ben sald'\`e gagliardi\\
\xx
Nella citta dal mar'intorno centa.
\end{verse}

\begin{remark}\label{RViete}
An equivalent, and somewhat quicker, way to obtain Cardano's formula is to use
\emph{Vi\`ete's substitution} $y=u-p/(3u)$
in $y^3+py+q=0$,
which yields \eqref{u6} directly. (This is obviously equivalent to setting
$y=u+v$ with $3uv=-p$ as above. See \cite[Chapter 3]{vdW} for Vi\`ete's
version of this.) 
\end{remark}

\section{Ferrari's solution of the quartic equation}\label{A4}

Consider again a fourth degree polynomial $g(y)=y^4+py^2+qy+r$ as in
\eqref{g4}.
The solution to  the equation $g(y)=0$ given in \refT{T4} is not the
solution originally found by Ferrari and presented by Cardano in 
\emph{Ars Magna} \cite[Chapter XXXIX]{ArsMagna} (\cf{} \refApp{Ahistory}).

Ferrari's method is as follows (in a modern version). 
From $y^4+py^2+qy+r=0$ we obtain, for any $z$,
\begin{equation}\label{F4z}
  (y^2+z)^2=y^4+2y^2z+z^2
=(2z-p)y^2-qy+z^2-r.
\end{equation}
We let $z\=(p+u)/2$ and obtain, for any $u$,
\begin{equation}\label{F4sala}
    \Bigpar{y^2+\frac{p+u}2}^2
=uy^2-qy+\frac{(p+u)^2}4-r.
\end{equation}
The \rhs{} is a quadratic polynomial in $y$, and its discriminant is
\begin{equation}
  q^2-4u\lrpar{\frac{(p+u)^2}4-r}
=-u^3-2pu^2-p^2u+4ru+q^2=-R(u),
\end{equation}
where $R$ is the cubic resolvent \eqref{R}.
Hence, if we choose $u$ as a non-zero root of $R$, 
then the \rhs{} of \eqref{F4sala} is the square of a linear polynomial.
More precisely, if we further let $\gam=\sqrt u$,
then the \rhs{} of \eqref{F4sala} is
\begin{equation}
  uy^2-qy+\frac{q^2}{4u}
=u\Bigpar{y-\frac{q}{2u}}^2
=\Bigpar{\gam y-\frac{q}{2\gam}}^2,
\end{equation}
and thus \eqref{F4sala} yields
\begin{equation}
  \Bigpar{y^2+\frac{p+u}2}^2
=\Bigpar{\gam y-\frac{q}{2\gam}}^2.
\end{equation}
Consequently,
\begin{equation}\label{F4pm}
  y^2+\frac{p+u}2
=\pm\Bigpar{\gam y-\frac{q}{2\gam}}.
\end{equation}
This yields a pair of quadratic equations in $y$,  whose solutions are
the four roots of $g(y)=0$. (It thus suffices to choose one non-zero root of
$R(u)=0$ in order to find all roots of $g(y)=0$. See \refR{R4q} below for a
justification.)

\begin{remark}
  \label{RFerrari}
Ferrari and Cardano considered, as said above, only equations with positive
coefficients (putting some of them on the \rhs), so they used different
versions of the method for different signs of our $p$, $q$ and $r$, but the
versions are essentially the same.

Moreover, in the original version, first $y^4+py^2$ is completed to a
square (usually, at least), yielding 
\begin{equation}\label{F4p0}
    \Bigpar{y^2+\frac{p}2}^2
=-qy+\frac{p^2}4-r;
\end{equation}
then this is further modified by considering 
$(y^2+p/2+t)^2$ and choosing $t$ so that the \rhs{} becomes a square.
This is obviously equivalent to the one-step completion of a square above,
with $z=p/2+t$. We further made the substitution $t=u/2$ in order to obtain
the same form of the cubic resolvent as before.

See \cite{Helfgott} for a detailed study of Cardano's  solutions
to quartics.
\end{remark}

We can connect Ferrari's method and the methods in \refS{S4} as follows,
using the notation in \refS{S4}. 
The roots $\gb_1,\gb_2$ are $\half\xgam2\pm\frac12(\xgam3+\xgam4)$, and are
thus the roots of the quadratic equation, using \eqref{v4}--\eqref{u+v+w} and
\eqref{gam234} and assuming $\xgam2\neq0$,
\begin{equation}
  \begin{split}
(2y-\xgam2)^2&=(\xgam3+\xgam4)^2
=v+w+2\xgam3\xgam4
=-2p-u-2q/\xgam2.
  \end{split}
\end{equation}
This equation can be rewritten, since $\xgam2=\sqrt u$,
\begin{gather}
  4y^2-4y\xgam2+2u+2p+2q/\xgam2=0, \\
4 y^2+2u+2p=4\xgam2y-2q/\xgam2, \\
y^2+\frac{u+p}2=\xgam2y - \frac{q}{2\xgam2}. \label{dper}
\end{gather}
The other two roots $\gb_3,\gb_4$ are obtained by replacing $\xgam2$ by
$-\xgam2$, the other square root of $u$.

We have thus obtained the equations \eqref{F4pm}, with $\gam=\xgam2\=\sqrt u$.

\begin{remark}  \label{R4q}
This derivation of \eqref{F4pm} from \refT{T4} shows clearly that the two
roots of each of the two quadratic equations in \eqref{F4pm} together yield
the four 
different roots of $g(y)=0$. Typically, the four roots are distinct and we
obtain all roots once each, 
but even when $g$ has multiple roots and there are repetitions in the roots
of \eqref{F4pm}, we obtain the roots of $g$ with correct
multiplicities from \eqref{F4pm}.
\end{remark}

\begin{remark}
  We started above with a reduced quartic $y^4+py^2+qy+r$ (as did Cardano
  and Ferrari), but, as noted by \citet[no.~27]{Lagrange}
 the method can also be applied directly to a general
  quartic $f(x)=x^4+bx^3+cx^2+dx+e$ by expanding $(x^2+\frac b2x+\sx)^2$ and
  using $f(x)=0$ in analogy with \eqref{F4z}; we then continue as above,
  obtaining a cubic resolvent equation for $\sx$, etc., see
  \cite[no.~27]{Lagrange} or
\cite[Section 12.1.C]{Cox} for details.
The resolvent equation for $\sx$ becomes
\begin{equation}\label{R3gen}
  \sx^3-\frac c2 \sx^2+\frac{bd-4e}4\sx+\frac{(4c-b^2)e-d^2}{8}=0.
\end{equation}

Comparing with \eqref{F4z}, and recalling $y=x+b/2$, 
we have $\sx=z+b^2/16=u/2+p/2+b^2/16$
and thus $u=2\sx-p-b^2/8$,
so the resulting cubic resovent equation \eqref{R3gen} is $R(2\sx-p-b^2/8)=0$,
with 
$R$ given by \eqref{R}.
Using \refR{Rbb1}, this can be written as $\tR(2\sx)=0$, as also follows from
\eqref{R3gen} and \eqref{tR}, so the roots of this resolvent equation
are simply $s_i/2$,
\ie, $\frac12(\ga_i\ga_j+\ga_k\ga_l)$ for permutations $ijkl$ of $1234$.
\end{remark}

\begin{remark}
  \label{Rbb2}
Expressed in the roots $\gb_i$, we have 
by \refR{Rbb}
\begin{equation}
  z=(u+p)/2=(\gb_1\gb_2+\gb_3\gb_4)/2.
\end{equation}
This also follows by \eqref{dper}, which implies
$\gb_1\gb_2=(u+p)/2+q/2\xgam2$ and, replacing $\xgam2$ by $-\xgam2$,
$\gb_3\gb_4=(u+p)/2-q/2\gam_1$. 
\end{remark}

\begin{remark}
  \label{RFpencil}
Ferrari's method has the following geometric interpretation in algebraic
geometry, see \cite{Faucette} and \cite{Auckly} for details.

Let $w\=y^2$. Then $(y,w)$ is a simultaneous solution of $w^2+pw+qy+r=0$ and
$w-y^2=0$, and thus also of the linear combination 
\begin{equation*}
  w^2+pw+qy+r+u(w-y^2)=0
\end{equation*}
for any $u$. As $u$ varies, this equation defines a family (called
\emph{pencil}) of quadratic curves (also known as \emph{conics}) in the
$(y,w)$-plane. A calculation essentially equivalent to the argument above
shows that this conic is singular, and thus a union of two lines, exactly
when $R(u)=0$, and then the two lines are $w+(p+u)/2=\pm(\gam y-q/2\gam)$
with $\gam=\sqrt u$ (assuming $u\neq0$), corresponding to \eqref{F4pm}.
Ferrari's method thus can be seen as finding one singular conic in the
pencil and decomposing it into a pair of lines; the solutions then are given
by the intersections between these lines and the conic $w=y^2$.
\end{remark}

\begin{remark}
  \label{Rdescartes}
Descartes gave in 1637 yet another method to solve quartic equations
(see \eg{} \cite{Helfgott}).
Descartes' s method is
based
on trying to factor $g(y)=(y^2+ky+l)(y^ 2+my+n)$ by identifying the
coefficients, which yields the equations 
\begin{align}
  k+m&=0,&
km+l+n&=p, &
kn+lm&=q, &
ln&=r.
\end{align}
This yields $m=-k$ and, after some algebra, $R(k^2)=0$, where $R$
is the cubic resolvent \eqref{R}. Hence we can solve $R(u)=0$, choose one
root $u$, let $k\=\gam\=\sqrt{u}$ and $m\=-k$ (the other square root of $u$);
solving for $l$ and $n$ then yields (for $u\neq0$)
\begin{equation}\label{descartes}
g(y)=\Bigpar{ y^2+\gam y+\frac{p+u}2-\frac{q}{2\gam}}
\Bigpar{ y^2-\gam y+\frac{p+u}2+\frac{q}{2\gam}}.
\end{equation}
Consequently, we see again that $g(y)=0$ is equivalent to \eqref{F4pm}.
\end{remark}

\newcommand\AMS{Amer. Math. Soc.}
\newcommand\Springer{Springer-Verlag}
\newcommand\Wiley{Wiley}

\newcommand\vol{\textbf}
\newcommand\jour{\emph}
\newcommand\book{\emph}
\newcommand\inbook{\emph}
\def\no#1#2,{\unskip#2, no. #1,} 
\newcommand\toappear{\unskip, to appear}

\newcommand\webcite[1]{\hfil  
   \penalty0 
\texttt{\def~{{\tiny$\sim$}}#1}\hfill\hfill}
\newcommand\webcitesvante{\webcite{http://www.math.uu.se/~svante/papers/}}
\newcommand\arxiv[1]{\webcite{arXiv:#1.}}
\newcommand\DOI[1]{DOI: \webcite{#1}}

\def\nobibitem#1\par{}


\begin{thebibliography}{99}

\bibitem{Auckly}
D. Auckly, 
Solving the quartic with a pencil.  
\emph{Amer. Math. Monthly}  \vol{114}  (2007),  no. 1, 29--39.

\bibitem{Bew}
J. Bewersdorff,
\emph{Algebra f\"ur Einsteiger}. 2nd ed.,
Friedr. Vieweg \& Sohn Verlag, Wiesbaden, 2004.
English transl.:
\emph{Galois Theory for Beginners. A Historical Perspective},
AMS, Providence, R.I., 2006.

\bibitem{ArsMagna}
G. Cardano,
\emph{Artis magnae, sive de regulis algebraicis. Lib. unus. Qui \&
  totius operis de arithmetica, quod opus perfectum inscripsit, est in
  ordine decimus.}
(\emph{Ars Magna}.)
Johann Petreius, Nuremberg, 1545. 
English transl.:
\emph{The Great Art, or The Rules of Algebra}. Translated and edited by
T. R. Witmer, MIT Press, Cambridge, Mass., 1968.

\bibitem{Cox}
David A. Cox,
\emph{Galois Theory},
Wiley-Interscience, Hoboken, NJ, 2004. 

\bibitem{Euler}
Leonhard Euler,
De formis radicum aequationum cuiusque ordinis coniectatio, 
\jour{Commentarii Academiae Scientiarum Petropolitanae} \vol{6} (1738),
216--231.
(Presented to the St.\ Petersburg Academy on November 2, 1733.)
English transl.:
A conjecture on the forms of the roots of equations,
translated by J. Bell. 
\arxiv{0806.1927}

\bibitem{Faucette}
W. M. Faucette, 
A geometric interpretation of the solution of the general quartic
polynomial.  
\emph{Amer. Math. Monthly}  \vol{103}  (1996),  no. 1, 51--57.

\bibitem{Frazier}
M. Frazier, 
\emph{An Introduction to Wavelets through Linear Algebra}. 
Springer, New York, 1999. 


\bibitem{Garling}
D. J. H. Garling,
\emph{Galois Theory}.
Cambridge Univ. Press, Cambridge, 1986.

\bibitem{Grillet}
P.A. Grillet, \emph{Abstract Algebra}. 2nd ed., Springer, New York, 2007.


\bibitem{Helfgott}
H. Helfgott and M. Helfgott,
A modern vision of the work of Cardano and Ferrari on quartics.
\emph{Loci} (June 2009), MathDL, 
The Mathematical Association of America. 
\DOI{10.4169/loci003312}.

\bibitem{Henriquez}
G. Henriquez,
The graphical interpretation of the complex roots of cubic equations.
\emph{Amer. Math. Monthly} \vol{42} (1935), no. 6, 383--384.

\bibitem{Holmes}
G. C. Holmes,
The use of hyperbolic cosines in solving cubic polynomials.
\emph{Math. Gazette} \vol{86} (2002), no. 507, 473--477.

\bibitem{IrwinW}
F. Irwin and H. N. Wright,
Some properties of polynomial curves.
\emph{Ann. Math.} \vol{19} (1917), no. 2, 152--158.

\bibitem[Janson(2007)]{SJN5}
S. Janson,
Resultant and discriminant of polynomials. 
 Note N5, 2007. 
\webcite{http://www.math.uu.se/~svante/papers/\#NOTES}

\bibitem{MathDL:Tartaglia}
F. Katscher, 
How Tartaglia solved the cubic equation.
\emph{Loci: Convergence} \vol3 (2006). \hfill
\webcite{http://mathdl.maa.org/mathDL/46/?pa=content\&sa=viewDocument\&nodeId=2433}
retrieved August 2, 2010.

\bibitem{Tartaglia}
F. Katscher, 
\book{Die kubischen Gleichungen bei Nicolo Tartaglia.  
Die relevanten Textstellen aus seinen 
,,\emph{Quesiti et inventioni diverse}`` auf
deutsch \"ubersetzt und kommentiert.} 
Verlag der \"Osterreichischen Akademie der
Wissenschaften, Vienna, 2001. 

\bibitem[Lagrange(1770--1771)]{Lagrange}
J. L. Lagrange,
R\'eflexions sur la r\'esolution alg\'ebrique des \'equations,
\emph{%
Nouveaux M\'emoires de l'Acad\'emie royale des Sciences et Belles-Lettres de
Berlin},
1770--1771.
Reprinted in
\emph{\OE uvres de Lagrange, vol. 3}, pp. 205--421,
J.-A. Serret ed.,
Gauthier-Villars, Paris, 1869

\nobibitem{Mann}
H. B. Mann,
On the casus irreducibilis.
\emph{Amer. Math. Monthly} \vol{71} (1964), 289--290.

\bibitem{Nickalls:cubic}
R. W. D. Nickalls,
A new approach to solving the cubic: Cardan's solution revealed.
\emph{Math. Gazette} \vol{77} (1993), no. 480, 354--359.


\bibitem{MT}
J. J. O'Connor \& E. F. Robertson,
\emph{The MacTutor History of Mathematics archive}.
University of St Andrews, Scotland.
\webcite{http://www-history.mcs.st-andrews.ac.uk/}


\bibitem{Serre}
J.-P. Serre,
\emph{Repr\'esentations lin\'eaires des groupes finis}.
2nd ed., Hermann, Paris, 1967.

\bibitem{vdW1}
B. L. van der Waerden,
\emph{Algebra I} (German). 7th ed., Springer-Verlag, Berlin, 1966.
English transl:
\emph{Algebra, Vol. I}, Springer-Verlag, New York, 1991.

\bibitem[van der Waerden(1985)]{vdW}
B. L. van der Waerden,
\emph{A History of Algebra}. 
Springer-Verlag, Berlin, 1985.




\end{thebibliography}
\end{document}